\documentclass[11pt,reqno]{amsart}
\usepackage{}
\usepackage{a4wide}
\numberwithin{equation}{section}
\usepackage{mathrsfs}
\usepackage{amsfonts}
\usepackage{amsmath}
\usepackage{stmaryrd}
\usepackage{amssymb}
\usepackage{amsthm}
\usepackage{mathrsfs}
\usepackage{url}
\usepackage{amsfonts}
\usepackage{amscd}
\usepackage{indentfirst}
\usepackage{enumerate}
\usepackage{amsmath,amsfonts,amssymb,amsthm}
\usepackage{amsmath,amssymb,amsthm,amscd}
\usepackage{graphicx,mathrsfs}
 \usepackage{appendix}

\usepackage[colorlinks,linkcolor=blue, citecolor=blue]{hyperref}
\numberwithin{equation}{section}

\usepackage[numbers,sort&compress]{natbib}

\usepackage{esint}
\usepackage{graphicx}
\usepackage[dvipsnames]{xcolor}
\newcommand{\R}{\mathbb{R}}

\newcommand\e\varepsilon
\newtheorem{Thm}{Theorem}[section]
\newtheorem{Lem}{Lemma}[section]
\newtheorem{Prop}{Proposition}[section]
\newtheorem{Def}{Definition}[section]
\newtheorem{Cor}{Corollary}[section]
\newtheorem{Rem}{Remark}[section]

\begin{document}

\title[Pohozaev identities and Kelvin transformation of semilinear Grushin equation]
{Pohozaev identities and Kelvin transformation of semilinear Grushin equation}

\author[Y. Wei and X. Zhou]{Yawei Wei and Xiaodong Zhou}

\address[Yawei Wei]{School of Mathematical Sciences and LPMC, Nankai University, Tianjin 300071, China}
\email{weiyawei@nankai.edu.cn}

\address[Xiaodong Zhou]{School of Mathematical Sciences, Nankai University, Tianjin 300071, China}
\email{1120210030@mail.nankai.edu.cn}

\thanks{Acknowledgements: This work is supported by the NSFC under the grands 12271269 and the Fundamental Research Funds for the Central Universities.}





\begin{abstract} In this paper, we study Pohozaev identities, Kelvin transformation and their applications of semilinear Grushin equation. First, we establish two Pohozaev identities generated from translations and determine the location of the concentration point for solution of a kind of Grushin equation by such identities. Next, we establish Pohozaev identity generated from scaling and prove the nonexistence of nontrivial solutions of another kind of Grushin equation by such identity. Finally, we provide the change of Grushin operator by Kelvin transformation and obtain the decay rate of solution at infinity for a critical Grushin equation by Kelvin transformation.
\end{abstract}

\maketitle
\section{Introduction}

\setcounter{equation}{0}

In this paper, we will consider the following general degenerate semilinear elliptic equation
\begin{equation}\label{i1}
  -\Delta_\gamma u=f(z,u),\;\;~\mbox{in}~\Omega\subset\R^{N+l},
\end{equation}
let
\begin{equation}\label{zzxxyyy}
  F(z,u)=\int_0^uf(z,s)ds,
\end{equation}
where $\Delta_\gamma$ is well-known \textit{Grushin operator} given by
\begin{equation}\label{i2}
  \Delta_\gamma u(z)=\Delta_xu(z)+|x|^{2\gamma}\Delta_yu(z),
\end{equation}
$\Delta_x$ and $\Delta_y$ are the Laplace operators in the variable $x$ and $y$ respectively, with $z=(x,y)\in\R^N\times\R^l=\R^{N+l}$ and $N+l\geq3$. Here, $\gamma\geq0$ is a real number and
\begin{equation}\label{zzxxyy}
  N_\gamma=N+(1+\gamma)l
\end{equation}
is the appropriate homogeneous dimension. The nonlinearity $f:\R^{N+l}\times\R^1\rightarrow\R^1$ is a continuous function which satisfies some conditions.

Firstly, we establish two Pohozaev identities generated from translations of \eqref{i1} .
As a special case of equation \eqref{i1}, by using such identities we study the location of the concentration point for solution of the following problem:
\begin{equation}\label{p10}
\begin{cases}
-\varepsilon^2\Delta_\gamma u+ a(z)u=u^{q-1},\;\;u>0,~ \mbox{in} ~\R^{N+l},\\[2mm]
u\in H^{1,2}_\gamma(\R^{N+l}),
\end{cases}
\end{equation}
where $\varepsilon>0$ is a small parameter, $2<q<2_\gamma^*=\frac{2N_\gamma}{N_\gamma-2}$, $a(z)\in C^2(\R^{N+l})$ satisfies $0<a_0\leq a(z)\leq a_1$ in $\R^{N+l}$, and $H^{1,2}_\gamma(\R^{N+l})$ is a weighted Sobolev space will be defined later.

Secondly, for the following Dirichlet zero boundary value problem:
\begin{equation}\label{ijk4}
\begin{cases}
-\Delta_\gamma u=f(z,u),\;\;~&\mbox{in}~\Omega,\\[2mm]
u=0,\;\;~&\mbox{on}~\partial\Omega,
\end{cases}
\end{equation}
we establish Pohozaev identity generated from scaling. As a special case of \eqref{ijk4}, by using such identity we study the nonexistence of nontrivial solution of the following problem:
\begin{equation}\label{.ma9}
\begin{cases}
-\Delta_\gamma u=u^{b-1},~u>0,~ &\mbox{in} ~\Omega,\\[2mm]
u=0,~ &\mbox{on} ~\partial\Omega,
\end{cases}
\end{equation}
where~$b=\frac{2(N+l)}{N+l-2}$,~$N+l\geq3$~and~$\Omega$~is a bounded domain in~$\R^{N+l}$.

Thirdly, we provide Kelvin transformation. As a special case of \eqref{i1}, by using such transformation we study the decay rate of solution at infinity for the following critical problem:
\begin{equation}\label{zzzzzzzzzzzzzzz}
  -\Delta_\gamma u=|u|^{2_\gamma^*-2}u,\;\;~\mbox{in}~\R^{N+l}.
\end{equation}

It's worth noting that $\Delta_\gamma$ is elliptic for $x\neq0$ and degenerates on the manifold $\{0\}\times\R^l$.
When $\gamma>0$ is an integer, the vector fields $X_1=\frac{\partial}{\partial x_1}$, $\cdots$, $X_N=\frac{\partial}{\partial x_N}$, $Y_1=|x|^\gamma\frac{\partial}{\partial y_1}$, $\cdots$, $Y_l=|x|^\gamma\frac{\partial}{\partial y_l}$ satisfy the H\"ormander's condition and then $\Delta_\gamma$ is hypoelliptic. Geometrically, $\Delta_\gamma$ comes from a sub-Laplace operator on a nilpotent Lie group of step $\gamma+1$ by a submersion. Specific explanation of geometric framework can be consulted in \cite{Bauer1} by Bauer et al.

In the past few decades, the analysis of this class of partial differential operators becomes more and more popular. Initially, Grushin studied some properties of this class of operators which degenerate on a submanifold including regularity, hypoellipticity and solvability (cf. \cite{Grushin2,Grushin3}). Some authors have dedicated a special attention for problems involving the Grushin operator. In \cite{Berestycki8}, Alves and his collaborators studied the existence of nontrivial solution for a class of degenerate elliptic problems involving the Grushin operator by applying variational methods. Bauer et al. gave explicit expressions of the fundamental solution of higher step Grushin type operator and Kohn-Laplacian operator in \cite{Bauer1} and \cite{Bauer6} respectively. Monti et al. \cite{Monti9} studied Sobolev inequalities for weighted gradients, which is very important in embedding theorems. Monticelli et al. studied the maximum principles for weak solutions and classical solutions of degenerate elliptic equations in \cite{DDM10} and \cite{xj} respectively. Metafune et al. \cite{Lp12} proved $L^p$ estimates for Baouendi-Grushin operator. Liu and his collaborators concerned with a critical Grushin-type problem in \cite{chunhua7}. They proved that this problem has infinitely many positive multi-bubbling solutions with arbitrarily large energy and cylindrical symmetry by applying Lyapunov-Schmidt reduction argument. Chung proved some results on the nonexistence and multiplicity of weak solutions for a class of semilinear elliptic systems of two equations involving Grushin type operators in \cite{bib1000}. One could refer to a series of papers \cite{zs1,zs2,zs3,zs5}, \cite{zs6}, \cite{zs7}, \cite{zs8,zs9,zs10} and literature therein.

There are overwhelming results in the field of Pohozaev identities and Kelvin transformation for the following semilinear elliptic equation:
\begin{equation}\label{i3}
  -\Delta u=f(x,u),\;\;~\mbox{in}~\R^N.
\end{equation}

Pohozaev identities for several typical elliptic equations on bounded domains or unbounded domains were studied thoroughly in \cite{Cao14} by Cao and his collaborators. This is an overview article. After establishing this kind of identity, mathematicians mentioned in this review proved many existence and nonexistence results of the solutions by using Pohozaev identities. For example,
\begin{equation}\label{i4}
\begin{cases}
-\Delta u=\lambda f(u),~x\in\Omega,\\[2mm]
u|_{\partial\Omega}=0,
\end{cases}
\end{equation}
where $\Omega\subset\R^N$ is an any bounded open domain of $\R^N$ ($N\geq3$), $f(u):\R^1\rightarrow\R^1$ is a continuous function satisfies $f(0)=0$. If $u\in C^2(\Omega)\cap C^1(\overline{\Omega})$ is the solution of \eqref{i4}, then the following Pohozaev identity for scaling holds:
\begin{equation}\label{i5}
  \frac{1}{2}\int_{\partial\Omega}(x\cdot\nu)|\nabla u|^2dS=\frac{2-N}{N}\lambda\int_\Omega uf(u)dx+N\lambda\int_\Omega F(u)dx,
\end{equation}
where $F(u)=\displaystyle\int_0^uf(t)dt$ and $\nu=\nu(x)$ is the unit outer normal vector of the point $x\in\partial\Omega$.
Identity \eqref{i5} established by Pohozaev in \cite{Pohozaev15} initially is for the sake of prove \eqref{i4} only has trivial solution under some technical conditions. In \cite{Cao14}, the authors summarize other application of Pohozaev identities such as to prove some local uniqueness results of peak solutions for some nonlinear elliptic partial differential equations. Tri \cite{Tri4} dealt with the degenerate elliptic operator $L=\frac{\partial^2}{\partial x_1^2}+\phi^2(x_1)\frac{\partial^2}{\partial x_2^2}$ in $\R^2$ and proved nonexistence of solution for second order Grushin equation with Dirichlet boundary value condition by establishing corresponding Pohozaev identity. And the Kelvin transformation for the operator $G_k=\frac{\partial^2}{\partial x_1^2}+x_1^{2k}\frac{\partial^2}{\partial x_2^2}$ also be considered in \cite{Tri4}. Tri \cite{Tri5} highly summarized the researches in \cite{Tri4}. Here, we prove the situation on higher dimensional space.

For the Kelvin transformation on \eqref{i3}, Axler and his collaborators write it in \cite{Axler16} which is graduate texts in mathematics. Kelvin transformation performs a important role in harmonic function theorem analogous to that played by the transformation $f(z)\mapsto f(\frac{1}{z})$ in holomorphic function theorem. In \eqref{i3}, we assume that the solution $u\in C^2(\R^N)$ decay to 0 when $|x|\rightarrow\infty$. But we want to know the decay rate of $u$ at infinity. If we want to prove $u$ is algebraic decay at infinity, i.e.
\begin{equation}\label{i6}
  u(x)\rightarrow\frac{c}{|x|^\alpha},\;\;|x|\rightarrow\infty,
\end{equation}
where $c\neq0$ and $\alpha>0$. Taking the change of variable $x=\frac{y}{|y|^2}$, then \eqref{i6} turns into
\begin{equation}\label{i7}
  |y|^{-\alpha}u(\frac{y}{|y|^2})\rightarrow c,\;\;|y|\rightarrow 0.
\end{equation}
Let $w(y)=|y|^{-\alpha}u(\frac{y}{|y|^2})$, then the problem that to prove $u(x)$ is algebraic decay when $|x|\rightarrow\infty$ transforms into the problem that to prove $w(y)$ is continuous when $|y|\rightarrow0$. This is the result on Laplace equations that we could read it in some mathematical text books such as in \cite{bib11} and the references therein. Monti et al. \cite{zs4} studied entire positive solutions of critical semilinear Grushin type equations by Kelvin transformation. Yu investigated Liouville type theorem for nonlinear elliptic equation involving Grushin operators also by using Kelvin transformation in \cite{Yu11}. Yu \cite{Yu11} obtained the change of Grushin operator by Kelvin transformation in the sense of distribution. We obtain the change of Grushin operator by Kelvin transformation in the sense of functions.

Now we introduce some notations and preliminaries that will be used later on and state our main results that will be proved later on.
We define a the weighted Sobolev space $H^{1,2}_\gamma(\R^{N+l})$ as work space as follows,
\begin{equation}\label{p3}
  H^{1,2}_\gamma(\R^{N+l})=\left\{u\in L^2(\R^{N+l})\;|\;\frac{\partial u}{\partial x_i}, |x|^\gamma\frac{\partial u}{\partial y_j}\in L^2(\R^{N+l}), i=1,\cdots,N, j=1,\cdots,l.\right\}.
\end{equation}
If $u\in H^{1,2}_\gamma(\R^{N+l})$, we denote the gradient operator $\nabla_\gamma$ as follows:
\begin{equation}\label{p4}
  \nabla_\gamma u=(\nabla_xu, |x|^\gamma\nabla_yu)=(u_{x_1},\cdots,u_{x_N},|x|^\gamma u_{y_1},\cdots,|x|^\gamma u_{y_l}),
\end{equation}
and so
\begin{equation}\label{p5}
  |\nabla_\gamma u|^2=|\nabla_xu|^2+|x|^{2\gamma}|\nabla_yu|^2.
\end{equation}
We should note that $H^{1,2}_\gamma(\R^{N+l})$ is a Hilbert space, if we endow with the inner product by
\begin{equation}\label{p6}
  \langle u,v\rangle_\gamma=\int_{\R^{N+l}}\nabla_\gamma u \cdot \nabla_\gamma v+uvdz,
\end{equation}
and the corresponding norm
\begin{equation}\label{p7}
  \|u\|_\gamma=\big(\int_{\R^{N+l}}|\nabla_\gamma u|^2+|u|^2dz\big)^\frac{1}{2}
\end{equation}
is induced by inner product \eqref{p6}.

$D^{1,2}_\gamma(\R^{N+l})$ is another important weighted Sobolev space, which is the closure of $C_0^\infty(\R^{N+l})$ with the norm
\begin{equation}\label{p17}
  \|u\|_{D^{1,2}_\gamma(\R^{N+l})}=\big(\int_{\R^{N+l}}|\nabla_\gamma u|^2dz\big)^\frac{1}{2},\;\;\;\;\forall u\in C_0^\infty(\R^{N+l}).
\end{equation}
From \cite{Monti9} we know that the embedding $D^{1,2}_\gamma(\R^{N+l})\hookrightarrow L^{2^*_\gamma}(\R^{N+l})$ is continuous, that is there exists a $C>0$ such that
\begin{equation}\label{p18}
  \|u\|_{L^{2^*_\gamma}(\R^{N+l})}\leq C\|u\|_{D^{1,2}_\gamma(\R^{N+l})},\;\;\;\;\forall u\in D^{1,2}_\gamma(\R^{N+l}).
\end{equation}

We continue to use the above designations in \cite{Berestycki8} by Alves and his collaborators.
Next, we give several definitions.
We define a new distance on $\R^{N+l}$ which is related to the Grushin operator.

\begin{Def}\label{zrf1}
Let the following distance be the new distance on~$\R^{N+l}$:~
\begin{equation}\label{p8}
  d(z,z_\e)=\left(\frac{1}{(1+\gamma)^2}|x-x_\e|^{2+2\gamma}+|y-y_\e|^2\right)^{\frac{1}{2+2\gamma}}
\end{equation}
for~$z=(x,y),z_\e=(x_\e,y_\e)\in\R^{N+l}$,~and set
\begin{equation}\label{p9}
  \widetilde{B}_r(z_\e):=\{z=(x,y)\in\R^{N+l}|\;d(z,z_\e)<r\}.
\end{equation}
be a ball in the sense of this new distance.
\end{Def}

Based on the definition of $H^{1,2}_\gamma(\R^{N+l})$, we define several other weighted Sobolev spaces as follows.
\begin{Def}\label{zrf2}
When $\gamma>\frac{1}{2}$, we define
\begin{equation}\label{bee1}
  H^{1,2}_{\gamma-\frac{1}{2}}(\R^{N+l})=\left\{u\in L^2(\R^{N+l})\;|\;\frac{\partial u}{\partial x_i}, |x|^{\gamma-\frac{1}{2}}\frac{\partial u}{\partial y_j}\in L^2(\R^{N+l}), i=1,\cdots,N, j=1,\cdots,l.\right\}.
\end{equation}
When $\gamma>0$, we define
\begin{equation}\label{bee2}
\begin{aligned}
H^{2,2}_\gamma(\R^{N+l}):=\left\{u\in L^2(\R^{N+l})\;|\;\frac{\partial u}{\partial x_i}, \frac{\partial^2u}{\partial x_i\partial x_j}, |x|^\gamma\frac{\partial u}{\partial y_j}, |x|^{2\gamma}\frac{\partial^2u}{\partial y_h\partial y_k}, |x|^\gamma\frac{\partial^2u}{\partial x_i\partial y_h} \in L^2(\R^{N+l}),\right.\
\\ i,j=1,\cdots,N, h,k=1,\cdots,l.\bigg\}.
\end{aligned}
\end{equation}
\begin{equation}\label{bee3}
\begin{aligned}
W^{2,2}(\R^{N+l}):=\left\{u\in L^2(\R^{N+l})\;|\;\frac{\partial u}{\partial x_i}, \frac{\partial u}{\partial y_j}, \frac{\partial^2u}{\partial x_i\partial x_j}, \frac{\partial^2u}{\partial y_h\partial y_k}, \frac{\partial^2u}{\partial x_i\partial y_h} \in L^2(\R^{N+l}),\right.\
\\ i,j=1,\cdots,N, h,k=1,\cdots,l.\bigg\}.
\end{aligned}
\end{equation}
is a standard Sobolev space.
\end{Def}

In what follows, we give the definition of weak solution to problem \eqref{p10}.

\begin{Def}\label{zrf3}
We say that $u\in H^{1,2}_\gamma(\R^{N+l})$ is a weak solution to equation \eqref{p10}, if $u\geq0$ a.e. in $\R^{N+l}$, $u\not\equiv0$ and
\begin{equation}\label{p14}
  \int_{\R^{N+l}}\varepsilon^2\nabla_\gamma u\cdot\nabla_\gamma\varphi+a(z)u\varphi dz=\int_{\R^{N+l}}u^{q-1}\varphi dz,\;\;\;\;\forall\varphi\in H^{1,2}_\gamma(\R^{N+l}).
\end{equation}
\end{Def}
Similarly, we can give the definition of weak solution to other equation in other weighted Sobolev space.

\begin{Def}\label{zrf4}
We say that $u_\varepsilon$ is a single-peak solution of \eqref{p10} concentrated at $z_0=(x_0,y_0)$ if
 $u_\varepsilon$  satisfies  \\
\noindent \textup{(i)} $u_\e$ has a local maximum point $z_{\varepsilon}=(x_\varepsilon,y_\varepsilon)\in \R^{N+l}$ such that $z_\e\to z_0\in \R^{N+l}$,as $\varepsilon\rightarrow0$;  \\
\noindent \textup{(ii)} For any given $\tau>0$, there exists $R\gg 1$ such that
\begin{equation*}
|u_\e(z)|\leq \tau~\,\,\mbox{for}\,\,z\in \R^{N+l}\backslash \widetilde{B}_{R\e}(z_\e);
\end{equation*}
\noindent \textup{(iii)} There exists $M>0$ such that
\begin{equation*}
u_\e\leq M.
\end{equation*}
\end{Def}

\begin{Def}\label{zrf5}
We say ~$\Omega$~is starshaped with respect to the point~$(0,0)$~if it holds
$$z\cdot\nu>0,\;\;~\mbox{for all}~z\in\partial\Omega.$$ And a domain $\Omega$ is called $L$-starshaped with respect to the point $(x,y)=(0,0)$ if the inequality
$$\left(\nu_x^2+|x|^{2\gamma}\nu_y^2\right)\left(x\cdot\nu_x+y\cdot\nu_y\right)>0$$ holds almost everywhere on $\partial\Omega$, where $\nu=(\nu_x,\nu_y)$ is the unit outward normal of the point of $\partial\Omega$.
\end{Def}

Our main results can be stated as follows.

\begin{Thm}\label{ma1}
If $u\in H^{1,2}_\gamma(\Omega)$ is the solution of equation \eqref{i1}, then $u$ satisfies
\begin{equation}\label{.ma2}
\begin{aligned}
\frac{1}{2}\int_{\partial D}|\nabla_\gamma u|^2\nu_x^idS-&\int_{\partial D}\left(\frac{\partial u}{\partial \nu_x}+|x|^{2\gamma}\frac{\partial u}{\partial \nu_y}\right)\frac{\partial u}{\partial x_i}dS-\gamma\int_D|\nabla_y u|^2|x|^{2(\gamma-1)}x_idz
\\=&\int_{\partial D}F(z,u)\nu_x^idS-\int_D\frac{\partial F(z,u)}{\partial x_i}dz,\;\;\;\;~\mbox{i=1,\ldots,N},
\end{aligned}
\end{equation}
and
\begin{equation}\label{.ma3}
\begin{aligned}
\\&\frac{1}{2}\int_{\partial D}|\nabla_\gamma u|^2\nu_y^jdS-\int_{\partial D}\left(\frac{\partial u}{\partial \nu_x}+|x|^{2\gamma}\frac{\partial u}{\partial \nu_y}\right)\frac{\partial u}{\partial y_j}dS
\\=&\int_{\partial D}F(z,u)\nu_y^jdS-\int_D\frac{\partial F(z,u)}{\partial y_j}dz,\;\;\;\;~\mbox{j=1,\ldots,l},
\end{aligned}
\end{equation}
where $D\subset\Omega$ is any domain of $\Omega$ and $\nu=(\nu_x,\nu_y)$ is the unit outward normal of the point of $\partial D$.
\end{Thm}

\begin{Cor}\label{th1}
Assume that $\gamma>\frac{1}{2}$ and let $z_\e$ is the local maximum point of the solution $u_\e\in H^{1,2}_{\gamma-\frac{1}{2}}(\R^{N+l})\cap H^{2,2}_\gamma(\R^{N+l})\cap W^{2,2}(\R^{N+l})\subset H^{1,2}_\gamma(\R^{N+l})$ of \eqref{p10} satisfying Definition \ref{zrf4}.
If $a(z)$ has critical points in the form of $(0,y)$, then $z_0=(x_0,y_0)$ satisfies $\nabla a(z_0)=0$ and $x_0=0$.
\end{Cor}

\begin{Thm}\label{300}
If $u\in H^{1,2}_\gamma(\Omega)$ is the solution of equation \eqref{ijk4},
then $u$ satisfies
\begin{equation}\label{.ma7}
\begin{aligned}
  (N+l)\int_\Omega F(z,u)dz&=-\int_\Omega z\cdot\nabla_z F(z,u)dz+\left(\frac{N+l}{2}-1\right)\int_\Omega f(z,u)udz
+\gamma\int_\Omega|x|^{2\gamma}|\nabla_yu|^2dz\\&+\frac{1}{2}\int_{\partial\Omega}\left(\frac{\partial u}{\partial\nu}\right)^2\left(\nu_x^2+|x|^{2\gamma}\nu_y^2\right)\left(x\cdot\nu_x+y\cdot\nu_y\right)dS,
\end{aligned}
\end{equation}
where $\Omega\subset\R^{N+l}$ is any domain of $\R^{N+l}$ and $\nu=(\nu_x,\nu_y)$ is the unit outward normal of the point of $\partial\Omega$.
\end{Thm}

\begin{Cor}\label{th1.1}
Suppose that~$\Omega$~is~$L$-starshaped defined in Definition \ref{zrf5} with respect to the point~$(0,0)$. In addition suppose that~$\frac{\partial u}{\partial\nu}\neq0$, where $\nu=(\nu_x,\nu_y)$ is the unit outward normal of $\partial\Omega$.~Then there does not exist a nontrivial positive solution for the problem \eqref{.ma9}.
\end{Cor}

\begin{Thm}\label{th2}
For $w(\widetilde{z})=d^{N_\gamma-2}(z,0)u(z)$, one has
\begin{equation*}
  \Delta_\gamma^{(\widetilde{z})}w(\widetilde{z})=d^{N_\gamma+2}(z,0)\Delta_\gamma u(z),
\end{equation*}
where $z=\left(\frac{1}{d^2(\widetilde{z},0)},\frac{1}{d^{2+2\gamma}(\widetilde{z},0)}\right)\widetilde{z}$.\\
The change $w(\widetilde{z})=d^{N_\gamma-2}(z,0)u(z)$ is called the Kelvin transformation, where $N_\gamma$ is defined as \eqref{zzxxyy}.
\end{Thm}

\begin{Cor}\label{th3}
Let $u(z)\in D^{1,2}_\gamma(\R^{N+l})$ be a solution of problem \eqref{zzzzzzzzzzzzzzz}. Then we have the estimate
\begin{equation}\label{p15}
  |u(z)|\leq \frac{C}{d^{N_\gamma-2}(z,0)},\;\;~\mbox{as}~d(z,0)\rightarrow\infty.
\end{equation}
In which, $N_\gamma$ is defined as \eqref{zzxxyy}.
\end{Cor}

The main contributions of this paper are summarized in the following three points. First, since there is no symmetric invariance for the $x$ variable in Grushin operator \eqref{i2}, we obtain two Pohozaev identities generated from translations by multiplying the equation \eqref{i1} by $\frac{\partial u}{\partial x_i}$ and $\frac{\partial u}{\partial y_j}$ respectively and then integrate on $D\subset\Omega$. Second, we obtain Pohozaev identity generated from scaling basically due to multiplied by $(z\cdot\nabla_zu)$ in the both side of Dirichlet problem \eqref{ijk4} to overcome the challenge from the asymmetry of the Grushin operators. Third, we calculate the change of Grushin operator by Kelvin transformation in the sense of functions. Furthermore, the applications of Pohozaev identities and the Kelvin transformation to nonlinear Grushin type elliptic equations are given.

The paper in the sequel is organized as follows. In section 2, we prove Theorem \ref{ma1} to obtain two Pohozaev identities generated from translations of \eqref{i1} and prove Corollary \ref{th1} to obtain the location of the concentration point for solution of \eqref{p10} satisfying Definition \ref{zrf4}. In section 3, we prove Theorem \ref{300} to obtain Pohozaev identity generated from scaling of \eqref{ijk4} and prove Corollary \ref{th1.1} to obtain
the nonexistence of nontrivial solutions for the problem \eqref{.ma9} under some assumptions. In section 4, we prove Theorem \ref{th2} to obtain the change of Grushin operator by Kelvin transformation and prove Corollary \ref{th3} to obtain the decay rate of solution at infinity for the problem \eqref{zzzzzzzzzzzzzzz}.

\section{Pohozaev identities generated from translations}

In this section, we prove two Pohozaev identities generated from translations of \eqref{i1} in Theorem \ref{ma1}. Then we determine the location of the concentration point for solution of \eqref{p10} satisfying Definition \ref{zrf4} by such identities in Corollary \ref{th1}.

\begin{proof}[\textbf{Proof of Theorem \ref{ma1}:}]

Multiplying the equation \eqref{i1} by $\frac{\partial u}{\partial x_i}$ and integrating on $D$, we obtain
\begin{equation}\label{ma4}
  \int_D-\Delta_\gamma u \frac{\partial u}{\partial x_i}dz=\int_Df(z,u)\frac{\partial u}{\partial x_i}dz,
\end{equation}
i.e.
\begin{equation}\label{ma5}
  \int_D\left(-\Delta_xu-|x|^{2\gamma}\Delta_yu\right)\frac{\partial u}{\partial x_i}dz=\int_Df(z,u)\frac{\partial u}{\partial x_i}dz.
\end{equation}
We denote the derivative of $F$ with respect to all of $x_i$ by $F_{x_i}(z,u)$,
and denote the derivative of $F$ with respect to $x_i$ in the first variable $z$ by $\frac{\partial F(z,u)}{\partial x_i}$. Using chain rule of derivation, we have
\begin{equation}\label{ma6}
  F_{x_i}(z,u)=\frac{\partial F(z,u)}{\partial x_i}+f(z,u)\frac{\partial u}{\partial x_i}.
\end{equation}
According to \eqref{ma6} and integration by parts, we first look at the RHS of \eqref{ma5}. For the convenience of readers, integration by parts is abbreviated as $I$ and Green first formula is abbreviated as $G$ in our process of calculation.
\begin{equation}\label{ma7}
\begin{aligned}
  &\int_Df(z,u)\frac{\partial u}{\partial x_i}dz
  \\\overset{\eqref{ma6}}=&\int_DF_{x_i}(z,u)-\frac{\partial F(z,u)}{\partial x_i}dz
  \\\overset{I}=&\int_{\partial D}F(z,u)\nu_x^idS-\int_D\frac{\partial F(z,u)}{\partial x_i}dz.
\end{aligned}
\end{equation}
According to Green first formula and integration by parts, then we look at the LHS of \eqref{ma5}.
First term:
\begin{equation}\label{ma8}
\begin{aligned}
  &\int_D-\Delta_xu\frac{\partial u}{\partial x_i}dz
  \\\overset{G}=&\int_D\nabla_x\frac{\partial u}{\partial x_i}\cdot\nabla_xudz-\int_{\partial D}\frac{\partial u}{\partial x_i}\frac{\partial u}{\partial \nu_x}dS
  \\=&\int_D\big(\frac{1}{2}|\nabla_xu|^2\big)_{x_i}dz-\int_{\partial D}\frac{\partial u}{\partial x_i}\frac{\partial u}{\partial \nu_x}dS
  \\\overset{I}=&\frac{1}{2}\int_{\partial D}|\nabla_xu|^2\nu_x^idS-\int_{\partial D}\frac{\partial u}{\partial x_i}\frac{\partial u}{\partial \nu_x}dS.
\end{aligned}
\end{equation}
Second term:
\begin{equation}\label{ma9}
\begin{aligned}
  &\int_D-|x|^{2\gamma}\frac{\partial u}{\partial x_i}\Delta_yudz
  \\\overset{G}=&\int_D\nabla_y\big(|x|^{2\gamma}\frac{\partial u}{\partial x_i}\big)\cdot\nabla_yudz-\int_{\partial D}|x|^{2\gamma}\frac{\partial u}{\partial x_i}\frac{\partial u}{\partial\nu_y}dS
  \\=&\int_D|x|^{2\gamma}\big(\frac{1}{2}|\nabla_yu|^2\big)_{x_i}dz-\int_{\partial D}|x|^{2\gamma}\frac{\partial u}{\partial x_i}\frac{\partial u}{\partial\nu_y}dS
  \\\overset{I}=&-\gamma\int_D|\nabla_yu|^2|x|^{2(\gamma-1)}x_idz+\frac{1}{2}\int_{\partial D}|\nabla_yu|^2|x|^{2\gamma}\nu_x^idS-\int_{\partial D}|x|^{2\gamma}\frac{\partial u}{\partial x_i}\frac{\partial u}{\partial\nu_y}dS.
\end{aligned}
\end{equation}
According to the LHS of \eqref{ma5} is equal to the RHS of \eqref{ma5} and combining with \eqref{ma7}-\eqref{ma9}, we obtain
\begin{equation}\label{ma10}
\begin{aligned}
  &\frac{1}{2}\int_{\partial D}\big(|\nabla_xu|^2+|x|^{2\gamma}|\nabla_yu|^2\big)\nu_x^idS-\int_{\partial D}\big(\frac{\partial u}{\partial \nu_x}+|x|^{2\gamma}\frac{\partial u}{\partial\nu_y}\big)\frac{\partial u}{\partial x_i}dS
  \\-&\gamma\int_D|\nabla_yu|^2|x|^{2(\gamma-1)}x_idz
  =\int_{\partial D}F(z,u)\nu_x^idS-\int_D\frac{\partial F(z,u)}{\partial x_i}dz,
\end{aligned}
\end{equation}
i.e. \eqref{.ma2}. So we have proved \eqref{.ma2} now. Next we will prove \eqref{.ma3} by similar idea.

Multiplying the equation by $\frac{\partial u}{\partial y_j}$ and integrating on $D$, we obtain
\begin{equation}\label{ma11}
  \int_D-\Delta_\gamma u \frac{\partial u}{\partial y_j}dz=\int_Df(z,u)\frac{\partial u}{\partial y_j}dz,
\end{equation}
i.e.
\begin{equation}\label{ma12}
  \int_D\left(-\Delta_xu-|x|^{2\gamma}\Delta_yu\right)\frac{\partial u}{\partial y_j}dz=\int_Df(z,u)\frac{\partial u}{\partial y_j}dz.
\end{equation}
We denote the derivative of $F$ with respect to all of $y_j$ by $F_{y_j}(z,u)$,
and denote the derivative of $F$ with respect to $y_j$ in the first variable $z$ by $\frac{\partial F(z,u)}{\partial y_j}$. Using chain rule of derivation, we have
\begin{equation}\label{ma13}
  F_{y_j}(z,u)=\frac{\partial F(z,u)}{\partial y_j}+f(z,u)\frac{\partial u}{\partial y_j}.
\end{equation}
According to \eqref{ma13} and integration by parts, we first look at the RHS of \eqref{ma12}.
\begin{equation}\label{ma14}
\begin{aligned}
  &\int_Df(z,u)\frac{\partial u}{\partial y_j}dz
  \\\overset{\eqref{ma13}}=&\int_DF_{y_j}(z,u)-\frac{\partial F(z,u)}{\partial y_j}dz
  \\\overset{I}=&\int_{\partial D}F(z,u)\nu_y^jdS-\int_D\frac{\partial F(z,u)}{\partial y_j}dz.
\end{aligned}
\end{equation}
According to Green first formula and integration by parts, then we look at the LHS of \eqref{ma12}.
First term:
\begin{equation}\label{ma15}
\begin{aligned}
  &\int_D-\Delta_xu\frac{\partial u}{\partial y_j}dz
  \\\overset{G}=&\int_D\nabla_x\frac{\partial u}{\partial y_j}\cdot\nabla_xudz-\int_{\partial D}\frac{\partial u}{\partial y_j}\frac{\partial u}{\partial \nu_x}dS
  \\=&\int_D\big(\frac{1}{2}|\nabla_xu|^2\big)_{y_j}dz-\int_{\partial D}\frac{\partial u}{\partial y_j}\frac{\partial u}{\partial \nu_x}dS
  \\\overset{I}=&\frac{1}{2}\int_{\partial D}|\nabla_xu|^2\nu_y^jdS-\int_{\partial D}\frac{\partial u}{\partial y_j}\frac{\partial u}{\partial \nu_x}dS.
\end{aligned}
\end{equation}
Second term:
\begin{equation}\label{ma16}
\begin{aligned}
  &\int_D-|x|^{2\gamma}\frac{\partial u}{\partial y_j}\Delta_yudz
  \\\overset{G}=&\int_D\nabla_y\big(|x|^{2\gamma}\frac{\partial u}{\partial y_j}\big)\cdot\nabla_yudz-\int_{\partial D}|x|^{2\gamma}\frac{\partial u}{\partial y_j}\frac{\partial u}{\partial\nu_y}dS
  \\=&\int_D|x|^{2\gamma}\big(\frac{1}{2}|\nabla_yu|^2\big)_{y_j}dz-\int_{\partial D}|x|^{2\gamma}\frac{\partial u}{\partial y_j}\frac{\partial u}{\partial\nu_y}dS
  \\\overset{I}=&\frac{1}{2}\int_{\partial D}|\nabla_yu|^2|x|^{2\gamma}\nu_y^jdS-\int_{\partial D}|x|^{2\gamma}\frac{\partial u}{\partial y_j}\frac{\partial u}{\partial\nu_y}dS.
\end{aligned}
\end{equation}
According to the LHS of \eqref{ma12} is equal to the RHS of \eqref{ma12} and combining with \eqref{ma14}-\eqref{ma16}, we obtain
\begin{equation}\label{ma17}
\begin{aligned}
  &\frac{1}{2}\int_{\partial D}\big(|\nabla_xu|^2+|x|^{2\gamma}|\nabla_yu|^2\big)\nu_y^jdS-\int_{\partial D}\big(\frac{\partial u}{\partial \nu_x}+|x|^{2\gamma}\frac{\partial u}{\partial\nu_y}\big)\frac{\partial u}{\partial y_j}dS
  \\=&\int_{\partial D}F(z,u)\nu_y^jdS-\int_D\frac{\partial F(z,u)}{\partial y_j}dz,
\end{aligned}
\end{equation}
i.e.\eqref{.ma3}. So we have proved \eqref{.ma3} now.
\end{proof}

\begin{Rem}\label{2}
Note that the equation we concern about is in the form of
\begin{equation}\label{.ma4}
  -\varepsilon^2\Delta_\gamma u+a(z)u=u^{q-1},\;\;~\mbox{in}~\R^{N+l}.
\end{equation}
So if $u$ is the solution of this equation, then $u$ satisfies
\begin{equation}\label{.ma50}
\begin{aligned}
&\frac{1}{2}\varepsilon^2\int_{\partial D}|\nabla_\gamma u|^2\nu_x^idS-\varepsilon^2\int_{\partial D}\left(\frac{\partial u}{\partial \nu_x}+|x|^{2\gamma}\frac{\partial u}{\partial \nu_y}\right)\frac{\partial u}{\partial x_i}dS-\varepsilon^2\gamma\int_D|\nabla_y u|^2|x|^{2(\gamma-1)}x_idz
\\=&\frac{1}{q}\int_{\partial D}u^q\nu_x^idS-\frac{1}{2}\int_{\partial D}a(z)u^2\nu_x^idS+\frac{1}{2}\int_D\frac{\partial a(z)}{\partial x_i}u^2dz,\;\;\;\;~\mbox{i=1,\ldots,N},
\end{aligned}
\end{equation}
and
\begin{equation}\label{.ma60}
\begin{aligned}
&\frac{1}{2}\varepsilon^2\int_{\partial D}|\nabla_\gamma u|^2\nu_y^jdS-\varepsilon^2\int_{\partial D}\left(\frac{\partial u}{\partial \nu_x}+|x|^{2\gamma}\frac{\partial u}{\partial \nu_y}\right)\frac{\partial u}{\partial y_j}dS
\\=&\frac{1}{q}\int_{\partial D}u^q\nu_y^jdS-\frac{1}{2}\int_{\partial D}a(z)u^2\nu_y^jdS+\frac{1}{2}\int_D\frac{\partial a(z)}{\partial y_j}u^2dz,\;\;\;\;~\mbox{j=1,\ldots,l}.
\end{aligned}
\end{equation}
We just take
\begin{equation*}
  F(z,u)=\frac{1}{\varepsilon^2}\left(\frac{1}{q}u^q-\frac{1}{2}a(z)u^2\right)
\end{equation*}
insert \eqref{.ma2} and \eqref{.ma3} respectively is alright.
\end{Rem}

From now on, we locate the position of the concentration point $z_0$ for solution of \eqref{p10} satisfying Definition \ref{zrf4} by using Pohozaev identities \eqref{.ma50} and \eqref{.ma60}. Before that, we should estimate the behavior of solution and several surface integrals away from the explosion point $z_\varepsilon$ by using maximum principles for weak solutions of degenerate equations in \cite{DDM10} and $L^2$ estimate for Grushin operator in \cite{Lp12} respectively.

\begin{Lem}\label{jc2}(cf. \cite{DDM10})
Assume that a general linear partial differential operators of the form
\begin{equation}\label{xj4}
  L:=a_{ij}(x)D_{ij}+b_i(x)D_i+c(x)
\end{equation}
acting on a domain $\Omega\subset\R^N$ satisfies the following hypotheses:

\textup{$(1)$} $a_{ij}=a_{ji}\in C^0(\overline{\Omega},\R)$ and $D_ja_{ij}\in L^\infty(\Omega,\R)$, $i,j=1,\cdots,N$.

\textup{$(2)$} $L$ has non-negative characteristic form.

\textup{$(3)$} The degeneracy set has Lebesgue measure zero.

\textup{$(4)$} L admits a uniformly elliptic direction.

\textup{$(5)$} $c(x)\in L^\infty(\Omega,\R)$.

\textup{$(6)$} $(\sqrt{A(x)})^{-1}\widetilde{b}(x)\in L^\infty(\Omega,\R^N)$, where $A(x)=[(a_{ij})]$, and $\widetilde{b}(x):=(\widetilde{b_1}(x),\cdots,\widetilde{b_N}(x))$ is defined by $\widetilde{b_i}(x)=D_ja_{ij}-b_i(x)$ for $i=1,\cdots,N$.

\textup{$(7)$} $1-\|(\sqrt{A(x)})^{-1}\widetilde{b}(x)\|_{L^\infty(\Omega)}C_p-C_p^2\sup_\Omega c(x)>0$, where $C_p>0$ denotes the best constant in the Poincar\'{e} inequality. \\
Then if $u\in H^{1,A}(\Omega):=\{u\in L^2(\Omega)\;|\;\sqrt{A}\nabla u\in L^2(\Omega)\}$ is a weak solution of $Lu\geq0$ in $\Omega$ and if $u\leq0$ on $\partial\Omega$, one has $u\leq0$ almost everywhere in $\Omega$.
\end{Lem}

\begin{Rem}
For $L=\Delta_x+|x|^{2\gamma}\Delta_y$ we study, where $(x,y)\in\R^{N+l}$, $\gamma>0$, it satisfies \textup{$(1)$}-\textup{$(6)$}, which can be referred to in \cite{xj}. So if \textup{$(7)$} satisfied in corresponding region, we can apply Lemma \ref{jc2} for $u\in H^{1,2}_\gamma(\R^{N+l})$.
\end{Rem}
\begin{Lem}\label{jc1}(cf. \cite{Lp12})
Assume that $\phi:\R^N\rightarrow[0,+\infty)$ belongs to the reverse H\"{o}lder class $B_2(\R^N)\cap B_N(\R^N)$. $Lu=f$, where $L=\Delta_x+\phi(x)\Delta_y$, $(x,y)\in\R^{N+M}$, $f\in L^2(\R^{N+M})$, $u\in D(L):=\{u\in L^2(\R^{N+M})\;|\;\nabla_xu, D_{x_ix_j}u \in L^2(\R^{N+M}), \phi^{\frac{1}{2}}\nabla_yu, \phi D_{y_hy_k}u, \phi^{\frac{1}{2}}D_{x_iy_h}u \in L^2(\R^{N+M})\}$, for every $i,j=1,\cdots,N$, $h,k=1,\cdots,M$. Then one has
\begin{equation}\label{xj1}
  \|D_{x_ix_j}u\|_2+\|\phi D_{y_hy_k}u\|_2\leq C\|f\|_2,\;\;u\in D(L).
\end{equation}
And
\begin{equation}\label{xj2}
  \|\nabla_xu\|_2+\|\phi^{\frac{1}{2}}\nabla_yu\|_2\leq C(\|f\|_2+\|u\|_2),\;\;u\in D(L).
\end{equation}
Moreover, for $i=1,,N$, $h=1,\cdots,M$, one has
\begin{equation}\label{xj3}
  \|\phi^{\frac{1}{2}}D_{x_iy_h}u\|_2\leq C\|f\|_2,\;\;u\in D(L).
\end{equation}
\end{Lem}

\begin{Rem}
When $\phi(x)=|x|^\alpha$, $\alpha\geq0$, $|x|^\alpha\in B_\infty(\R^N)\subset B_p(\R^N)$, for $1<p<\infty$. So we can apply Lemma \ref{jc1} in $H^{2,2}_\gamma(\R^{N+l})$ and $W^{2,2}(\R^{N+l})$ and apply \eqref{xj2} in $H^{1,2}_{\gamma-\frac{1}{2}}(\R^{N+l})$ for $L=\Delta_x+|x|^{2\gamma}\Delta_y$ when $\gamma>\frac{1}{2}$.
\end{Rem}

\begin{Prop}\label{1010}
Assume that $\gamma>\frac{1}{2}$ and $u_\e\in H^{1,2}_{\gamma-\frac{1}{2}}(\R^{N+l})\cap H^{2,2}_\gamma(\R^{N+l})\cap W^{2,2}(\R^{N+l})\subset H^{1,2}_\gamma(\R^{N+l})$ is the solution of \eqref{p10} satisfying Definition \ref{zrf4}. Then for any $\alpha\in(0,\displaystyle\inf_{z\in \R^{N+l}}a(z))$, there are constants $\theta>0$ and $C>0$, such that
\begin{equation}\label{..4}
  u_\e(z)\leq Ce^{-\frac{\theta {d(z,z_\e)}}{\varepsilon}},\;\;\forall z\in\R^{N+l},
\end{equation}
where $C=(M+1)e^{\theta R}$.
Moreover, if we denote
\begin{eqnarray*}
  J_1 &=& \int_{\partial\widetilde{B}_\delta(z_\varepsilon)}|\nabla_\gamma u_\varepsilon|^2\nu_x^idS, \\
  J_2 &=& \int_{\partial\widetilde{B}_\delta(z_\varepsilon)}\left(\frac{\partial{u_\e}}{\partial \nu_x}+|x|^{2\gamma}\frac{\partial{u_\e}}{\partial \nu_y}\right)\frac{\partial{u_\e}}{\partial x_i}dS, \\
  J_3 &=& \int_{\partial\widetilde{B}_\delta(z_\e)}|\nabla_\gamma u_\e|^2\nu_y^jdS, \\
  J_4 &=& \int_{\partial\widetilde{B}_\delta(z_\e)}\left(\frac{\partial {u_\e}}{\partial \nu_x}+|x|^{2\gamma}\frac{\partial {u_\e}}{\partial \nu_y}\right)\frac{\partial {u_\e}}{\partial y_j}dS,
\end{eqnarray*}
then
\begin{equation}\label{yph23}
  J_i\leq\widetilde{C}\int_{\widetilde{B}_\delta(z_\varepsilon)}e^{-\frac{\theta {d(z,z_\e)}}{\varepsilon}}dz,\;\;i=1,\cdots,4,
\end{equation}
where $\delta>0$ and $\widetilde{C}>0$ is a certain constant.
\end{Prop}

\begin{proof}
Denote $a_m=\displaystyle\inf_{z\in \R^{N+l}}a(z)$. Write
\begin{equation*}
  -\varepsilon^2\Delta_\gamma u_\e+ a(z)u_\e=u_\e^{q-1}
\end{equation*}
as
\begin{equation*}
  \varepsilon^2\Delta_\gamma u_\e-\left(a(z)-u_\e^{q-2}\right)u_\e=0.
\end{equation*}
By \noindent \textup{(ii)} in Definition \ref{zrf4}, for any $\tau>0$, there exists $R>>1$, such that for $\forall z\in \R^{N+l}\backslash \widetilde{B}_{R\e}(z_\e)$,where $z_\e=(x_\e,y_\e)$, we have~$u_\varepsilon\leq\tau$.
So for any $\alpha\in(0,a_m)$, there exists $R>0$, such that for $\forall z\in \R^{N+l}\backslash \widetilde{B}_{R\e}(z_\e)$, we have
\begin{equation*}
  a(z)-u_\varepsilon^{q-2}\geq a_m-\tau^{q-2}\geq\alpha>0.
\end{equation*}
Thus in $\R^{N+l}\backslash \widetilde{B}_{R\e}(z_\e)$, we have
\begin{equation*}
  \varepsilon^2\Delta_\gamma u_\varepsilon-\alpha u_\varepsilon\geq0.
\end{equation*}
Let
\begin{equation}\label{..6}
  L_\varepsilon u=\varepsilon^2\Delta_\gamma u-\alpha u.
\end{equation}
Then we can verify $L_\varepsilon$ satisfies the item \textup{$(7)$} in Lemma \ref{jc2}. So we can apply maximum principles for weak solutions in $\R^{N+l}\backslash \widetilde{B}_{R\e}(z_\e)$.
On the one hand
\begin{equation}\label{..7}
  L_\varepsilon u_\varepsilon\geq0.
\end{equation}
On the other hand, if $\varepsilon$ small enough,
\begin{equation}\label{..8}
\begin{aligned}
  &L_\varepsilon e^{-\frac{\theta {d(z,z_\e)}}{\varepsilon}}
  \\=&e^{-\frac{\theta {d(z,z_\e)}}{\varepsilon}}\left(\theta^2\frac{1}{(1+\gamma)^2}|x|^{2\gamma}d^{-2\gamma}(z,z_\varepsilon)
  -\theta\varepsilon\frac{1}{(1+\gamma)^2}|x|^{2\gamma}d^{-1-2\gamma}(z,z_\varepsilon)(N_\gamma-1)-\alpha\right)
  \leq0.
\end{aligned}
\end{equation}
By \noindent \textup{(iii)} in Definition \ref{zrf4}, it holds $u_\varepsilon\leq M$ for some $M$. Let
\begin{equation*}
  \omega_\varepsilon=u_\varepsilon-(M+1)e^{\theta R}e^{-\frac{\theta {d(z,z_\e)}}{\varepsilon}}.
\end{equation*}
Then it follows from \eqref{..7} and \eqref{..8} that in $\R^{N+l}\backslash \widetilde{B}_{R\e}(z_\e)$,
\begin{equation}\label{..9}
  L_\varepsilon\omega_\varepsilon=L_\varepsilon u_\varepsilon-(M+1)e^{\theta R}L_\varepsilon e^{-\frac{\theta {d(z,z_\e)}}{\varepsilon}}\geq0.
\end{equation}
But on $\partial\widetilde{B}_{R\e}(z_\e)$,
\begin{equation*}
  \omega_\varepsilon=u_\varepsilon-(M+1)\leq0.
\end{equation*}
Hence, by Lemma \ref{jc2}, we have
\begin{equation*}
  \omega_\varepsilon\leq0, \;\; ~\mbox{in}~\R^{N+l}\backslash \widetilde{B}_{R\e}(z_\e),
\end{equation*}
i.e.
\begin{equation*}
  u_\varepsilon\leq Ce^{-\frac{\theta {d(z,z_\e)}}{\varepsilon}},
\end{equation*}
where~$C=(M+1)e^{\theta R}$.~
While in $\widetilde{B}_{R\e}(z_\e)$,
\begin{equation*}
  (M+1)e^{\theta R}e^{-\frac{\theta {d(z,z_\e)}}{\varepsilon}}\geq M+1>u_\varepsilon.
\end{equation*}
So we complete the proof of \eqref{..4}.

Next we prove another part of proposition. By using \eqref{..4} and $L^2$ estimate to \eqref{p10}, we estimate the following surface integrals on $\partial\widetilde{B}_\delta(z_\varepsilon)$, where $\delta>0$.

For $J_1$,
\begin{equation}\label{yph1}
\begin{aligned}
&\int_{\partial\widetilde{B}_\delta(z_\varepsilon)}|\nabla_\gamma u_\varepsilon|^2\nu_x^idS
\\=&\int_{\widetilde{B}_\delta(z_\varepsilon)}\frac{\partial}{\partial x_i}(|\nabla_\gamma u_\varepsilon|^2)dz
\\=&\sum_{j=1}^N\int_{\widetilde{B}_\delta(z_\varepsilon)}2{u_\varepsilon}_{x_j}{u_\varepsilon}_{x_jx_i}dz
+\sum_{h=1}^l\int_{\widetilde{B}_\delta(z_\varepsilon)}2\gamma|x|^{2\gamma-2}x_i{u_\varepsilon}_{y_h}^2
+2|x|^{2\gamma}{u_\varepsilon}_{y_h}{u_\varepsilon}_{y_hx_i}dz
\\\leq&\sum_{j=1}^N\left(\int_{\widetilde{B}_\delta(z_\varepsilon)}{u_\varepsilon}_{x_j}^2
+{u_\varepsilon}_{x_jx_i}^2dz\right)
\\+&\sum_{h=1}^l\left(\int_{\widetilde{B}_\delta(z_\varepsilon)}2\gamma|x|^{2\gamma-1}{u_\varepsilon}_{y_h}^2dz
+\int_{\widetilde{B}_\delta(z_\varepsilon)}|x|^{2\gamma}{u_\varepsilon}_{y_h}^2dz
+\int_{\widetilde{B}_\delta(z_\varepsilon)}|x|^{2\gamma}{u_\varepsilon}_{y_hx_i}^2dz\right),
\end{aligned}
\end{equation}
in which, according to \eqref{xj2},
\begin{equation}\label{yph2}
\begin{aligned}
  &\int_{\widetilde{B}_\delta(z_\varepsilon)}|\nabla_xu_\varepsilon|^2dz
  +\int_{\widetilde{B}_{\frac{\delta}{2}}(z')}|x|^{2\gamma}|\nabla_yu_\varepsilon|^2dz
  \\\leq&C\left(\frac{1}{\varepsilon^2}\|a(z)u_\varepsilon-u_\varepsilon^{q-1}\|_{L^2(\widetilde{B}_\delta(z_\e))}
  +\|u_\varepsilon\|_{L^2(\widetilde{B}_\delta(z_\e))}\right)^2
  \\\leq&C\left(\frac{1}{\varepsilon^2}a_1\left(\int_{\widetilde{B}_\delta(z_\varepsilon)}e^{-\frac{2\theta {d(z,z_\e)}}{\varepsilon}}dz\right)^\frac{1}{2}+\frac{1}{\varepsilon^2}
  \left(\int_{\widetilde{B}_\delta(z_\varepsilon)}e^{-\frac{2(q-1)\theta {d(z,z_\e)}}{\varepsilon}}dz\right)^\frac{1}{2}+\left(\int_{\widetilde{B}_\delta(z_\varepsilon)}e^{-\frac{2\theta {d(z,z_\e)}}{\varepsilon}}dz\right)^\frac{1}{2}\right)^2
    \\\leq&\widetilde{C}\left(\left(\int_{\widetilde{B}_\delta(z_\varepsilon)}e^{-\frac{2\theta {d(z,z_\e)}}{2\varepsilon}}dz\right)^\frac{1}{2}+
  \left(\int_{\widetilde{B}_\delta(z_\varepsilon)}e^{-\frac{2(q-1)\theta {d(z,z_\e)}}{2\varepsilon}}dz\right)^\frac{1}{2}+\left(\int_{\widetilde{B}_\delta(z_\varepsilon)}e^{-\frac{2\theta {d(z,z_\e)}}{\varepsilon}}dz\right)^\frac{1}{2}\right)^2
  \\\leq&\widetilde{C}\int_{\widetilde{B}_\delta(z_\varepsilon)}e^{-\frac{\theta {d(z,z_\e)}}{\varepsilon}}dz,
\end{aligned}
\end{equation}
where we recall that $a(z)\leq a_1$ and $q>2$,\\
according to \eqref{xj1},
\begin{equation}\label{yph3}
\begin{aligned}
\int_{\widetilde{B}_\delta(z_\varepsilon)}{u_\varepsilon}_{x_jx_i}^2dz\leq &C\left(\frac{1}{\varepsilon^2}\|a(z)u_\varepsilon-u_\varepsilon^{q-1}\|_{L^2(\widetilde{B}_\delta(z_\varepsilon))}
  \right)^2
  \\\leq&\widetilde{C}\int_{\widetilde{B}_\delta(z_\varepsilon)}e^{-\frac{\theta {d(z,z_\e)}}{\varepsilon}}dz,
\end{aligned}
\end{equation}
according to \eqref{xj3},
\begin{equation}\label{yph4}
\begin{aligned}
  \int_{\widetilde{B}_\delta(z_\varepsilon)}|x|^{2\gamma}{u_\varepsilon}_{y_hx_i}^2dz\leq
  &C\left(\frac{1}{\varepsilon^2}\|a(z)u_\varepsilon-u_\varepsilon^{q-1}\|_{L^2(\widetilde{B}_\delta(z_\varepsilon))}
  \right)^2
  \\\leq&\widetilde{C}\int_{\widetilde{B}_\delta(z_\varepsilon)}e^{-\frac{\theta {d(z,z_\e)}}{\varepsilon}}dz,
\end{aligned}
\end{equation}
since $u\in H^{1,2}_{\gamma-\frac{1}{2}}(\R^{N+l})$, by using \eqref{xj2} with $\phi(x)=|x|^{2\gamma-1}$, we obtain
\begin{equation}\label{yph5}
\begin{aligned}
  \int_{\widetilde{B}_\delta(z_\varepsilon)}2\gamma|x|^{2\gamma-1}{u_\varepsilon}_{y_h}^2dz\leq
  &C\left(\frac{1}{\varepsilon^2}\|a(z)u_\varepsilon-u_\varepsilon^{q-1}\|_{L^2(\widetilde{B}_\delta(z_\varepsilon))}
  \right)^2
  \\\leq&\widetilde{C}\int_{\widetilde{B}_\delta(z_\varepsilon)}e^{-\frac{\theta {d(z,z_\e)}}{\varepsilon}}dz.
\end{aligned}
\end{equation}
From \eqref{yph1}-\eqref{yph5}, we have the estimate
\begin{equation}\label{yph6}
  J_1=\int_{\partial\widetilde{B}_\delta(z_\varepsilon)}|\nabla_\gamma u_\varepsilon|^2\nu_x^idS\leq\widetilde{C}\int_{\widetilde{B}_\delta(z_\varepsilon)}e^{-\frac{\theta {d(z,z_\e)}}{\varepsilon}}dz.
\end{equation}

Similarly, for $J_2$
\begin{equation}\label{yph8}
\begin{aligned}
&\int_{\partial\widetilde{B}_\delta(z_\varepsilon)}\left(\frac{\partial u_\varepsilon}{\partial \nu_x}+|x|^{2\gamma}\frac{\partial u_\varepsilon}{\partial \nu_y}\right)\frac{\partial u_\varepsilon}{\partial x_i}dS
 \\=&\int_{\partial\widetilde{B}_\delta(z_\varepsilon)}\left(\nabla_xu_\varepsilon\cdot\nu_x
 +|x|^{2\gamma}\nabla_yu_\varepsilon\cdot\nu_y
 \right)\frac{\partial u_\varepsilon}{\partial x_i}dS
 \\=&\sum_{j=1}^N\int_{\partial\widetilde{B}_\delta(z_\varepsilon)}{u_\varepsilon}_{x_j}{u_\e}_{x_i}\nu_x^jdS
    +\sum_{h=1}^l\int_{\partial\widetilde{B}_\delta(z_\varepsilon)}|x|^{2\gamma}{u_\e}_{y_h}{u_\e}_{x_i}\nu_y^hdS
 \\=&\sum_{j=1}^N\int_{\widetilde{B}_\delta(z_\varepsilon)}\frac{\partial}{\partial x_j}({u_\e}_{x_j}{u_\e}_{x_i})dz
    +\sum_{h=1}^l\int_{\widetilde{B}_\delta(z_\varepsilon)}\frac{\partial}{\partial y_h}(|x|^{2\gamma}{u_\e}_{y_h}{u_\e}_{x_i})dz
 \\=&\sum_{j=1}^N\int_{\widetilde{B}_\delta(z_\varepsilon)}{u_\e}_{x_jx_j}{u_\e}_{x_i}+{u_\e}_{x_j}{u_\e}_{x_ix_j}dz
 \\+&\sum_{h=1}^l\int_{\widetilde{B}_\delta(z_\varepsilon)}|x|^{2\gamma}{u_\e}_{y_hy_h}{u_\e}_{x_i}+|x|^{2\gamma}
 {u_\e}_{y_h}{u_\e}_{x_iy_h}dz,
\end{aligned}
\end{equation}
where
\begin{equation}\label{yph9}
\begin{aligned}
 &\int_{\widetilde{B}_\delta(z_\varepsilon)}{u_\e}_{x_jx_j}{u_\e}_{x_i}+{u_\e}_{x_j}{u_\e}_{x_ix_j}dz
 \\\leq&\frac{1}{2}\left(\int_{\widetilde{B}_\delta(z_\varepsilon)}{u_\e}_{x_jx_j}^2+{u_\e}_{x_i}^2+{u_\e}_{x_j}^2
 +{u_\e}_{x_ix_j}^2dz\right)
 \\\leq&\widetilde{C}\int_{\widetilde{B}_\delta(z_\varepsilon)}e^{-\frac{\theta {d(z,z_\e)}}{\varepsilon}}dz,
\end{aligned}
\end{equation}
this is due to \eqref{xj1} and \eqref{xj2} and similar to \eqref{yph2} and \eqref{yph3},\\
where
\begin{equation}\label{yph10}
\begin{aligned}
  &\int_{\widetilde{B}_\delta(z_\varepsilon)}|x|^{2\gamma}{u_\e}_{y_hy_h}{u_\e}_{x_i}
  +|x|^{2\gamma}{u_\e}_{y_h}{u_\e}_{x_iy_h}dz
  \\\leq&\frac{1}{2}\left(\int_{\widetilde{B}_\delta(z_\varepsilon)}|x|^{4\gamma}{u_\e}_{y_hy_h}^2+{u_\e}_{x_i}^2
  +|x|^{2\gamma}{u_\e}_{y_h}^2+|x|^{2\gamma}{u_\e}_{x_iy_h}^2dz\right)
  \\\leq&\widetilde{C}\int_{\widetilde{B}_\delta(z_\varepsilon)}e^{-\frac{\theta {d(z,z_\e)}}{\varepsilon}}dz,
\end{aligned}
\end{equation}
this is owing to \eqref{xj1}, \eqref{xj2} and \eqref{xj3}.\\
From \eqref{yph8}-\eqref{yph10}, we have the estimate
\begin{equation}\label{yph12}
    J_2=\int_{\partial\widetilde{B}_\delta(z_\varepsilon)}\left(\frac{\partial{u_\e}}{\partial \nu_x}+|x|^{2\gamma}\frac{\partial{u_\e}}{\partial \nu_y}\right)\frac{\partial{u_\e}}{\partial x_i}dS
  \leq \widetilde{C}\int_{\widetilde{B}_\delta(z_\varepsilon)}e^{-\frac{\theta {d(z,z_\e)}}{\varepsilon}}dz.
\end{equation}

For $J_3$,
\begin{equation}\label{yph13}
\begin{aligned}
  &\int_{\partial\widetilde{B}_\delta(z_\varepsilon)}|\nabla_\gamma u_\e|^2\nu_y^jdS
  \\=&\int_{\widetilde{B}_\delta(z_\varepsilon)}\frac{\partial}{\partial y_j}(|\nabla_\gamma u_\e|^2)dz
  \\=&\sum_{i=1}^N\int_{\widetilde{B}_\delta(z_\varepsilon)}2{u_\e}_{x_i}{u_\e}_{x_iy_j}dz
  +\sum_{h=1}^l\int_{\widetilde{B}_\delta(z_\varepsilon)}2|x|^{2\gamma}{u_\e}_{y_h}{u_\e}_{y_hy_j}dz.
\end{aligned}
\end{equation}
Since $u\in H^{2,2}_\gamma(\R^{N+l})\cap W^{2,2}(\R^{N+l})$, by using Lemma \ref{jc1} with $\phi(x)=|x|^{2\gamma}$ and $\phi(x)=1$, we obtain
\begin{equation}\label{yph14}
\begin{aligned}
  &\int_{\widetilde{B}_\delta(z_\varepsilon)}2{u_\e}_{x_i}{u_\e}_{x_iy_j}dz
 \\\leq&\int_{\widetilde{B}_\delta(z_\varepsilon)}{u_\e}_{x_i}^2+{u_\e}_{x_iy_j}^2dz
  \\\leq&\widetilde{C}\int_{\widetilde{B}_\delta(z_\varepsilon)}e^{-\frac{\theta {d(z,z_\e)}}{\varepsilon}}dz,
\end{aligned}
\end{equation}
and
\begin{equation}\label{yph15}
\begin{aligned}
  &\int_{\widetilde{B}_\delta(z_\varepsilon)}2|x|^{2\gamma}{u_\e}_{y_h}{u_\e}_{y_hy_j}dz
  \\\leq&\int_{\widetilde{B}_\delta(z_\varepsilon)}|x|^{4\gamma}{u_\e}_{y_hy_j}^2+{u_\e}_{y_h}^2dz
    \\\leq &\widetilde{C}\int_{\widetilde{B}_\delta(z_\varepsilon)}e^{-\frac{\theta {d(z,z_\e)}}{\varepsilon}}dz.
\end{aligned}
\end{equation}
From \eqref{yph13}-\eqref{yph15}, we have the esitimate
\begin{equation}\label{yph17}
  J_3=\int_{\partial\widetilde{B}_\delta(z_\e)}|\nabla_\gamma u_\e|^2\nu_y^jdS
  \leq\widetilde{C}\int_{\widetilde{B}_\delta(z_\varepsilon)}e^{-\frac{\theta {d(z,z_\e)}}{\varepsilon}}dz.
\end{equation}

For $J_4$,
\begin{equation}\label{yph18}
\begin{aligned}
  &\int_{\partial\widetilde{B}_\delta(z_\varepsilon)}\left(\frac{\partial u_\varepsilon}{\partial \nu_x}+|x|^{2\gamma}\frac{\partial u_\varepsilon}{\partial \nu_y}\right)\frac{\partial u_\varepsilon}{\partial y_j}dS
  \\=&\int_{\partial\widetilde{B}_\delta(z_\varepsilon)}\left(
  \nabla_xu_\varepsilon\cdot\nu_x+|x|^{2\gamma}\nabla_yu_\varepsilon\cdot\nu_y
  \right)\frac{\partial u_\varepsilon}{\partial y_j}dS
  \\=&\sum_{i=1}^N\int_{\partial\widetilde{B}_\delta(z_\varepsilon)}{u_\e}_{x_i}{u_\e}_{y_j}\nu_x^idS
  +\sum_{h=1}^l\int_{\partial\widetilde{B}_\delta(z_\varepsilon)}|x|^{2\gamma}{u_\e}_{y_h}{u_\e}_{y_j}\nu_y^hdS
  \\=&\sum_{i=1}^N\int_{\widetilde{B}_\delta(z_\varepsilon)}\frac{\partial}{\partial x_i}({u_\e}_{x_i}{u_\e}_{y_j})dz
  +\sum_{h=1}^l\int_{\widetilde{B}_\delta(z_\varepsilon)}\frac{\partial}{\partial y_h}(|x|^{2\gamma}{u_\e}_{y_h}{u_\e}_{y_j})dz
  \\=&\sum_{i=1}^N\int_{\widetilde{B}_\delta(z_\varepsilon)}{u_\e}_{x_ix_i}{u_\e}_{y_j}+{u_\e}_{x_i}{u_\e}_{y_jx_i}dz
  \\+&\sum_{h=1}^l\int_{\widetilde{B}_\delta(z_\varepsilon)}|x|^{2\gamma}{u_\e}_{y_hy_h}{u_\e}_{y_j}
  +|x|^{2\gamma}{u_\e}_{y_h}{u_\e}_{y_jy_h}dz.
\end{aligned}
\end{equation}
Since $u\in H^{2,2}_\gamma(\R^{N+l})\cap W^{2,2}(\R^{N+l})$, by using Lemma \ref{jc1} with $\phi(x)=|x|^{2\gamma}$ and $\phi(x)=1$, we obtain
\begin{equation}\label{yph19}
\begin{aligned}
  &\int_{\widetilde{B}_\delta(z_\varepsilon)}{u_\e}_{x_ix_i}{u_\e}_{y_j}+{u_\e}_{x_i}{u_\e}_{y_jx_i}dz
  \\\leq&\frac{1}{2}\left(\int_{\widetilde{B}_\delta(z_\varepsilon)}{u_\e}_{x_ix_i}^2+{u_\e}_{y_j}^2
  +{u_\e}_{x_i}^2+{u_\e}_{y_jx_i}^2dz\right)
  \\\leq&\widetilde{C}\int_{\widetilde{B}_\delta(z_\varepsilon)}e^{-\frac{\theta {d(z,z_\e)}}{\varepsilon}}dz,
\end{aligned}
\end{equation}
and
\begin{equation}\label{yph20}
\begin{aligned}
  &\int_{\widetilde{B}_\delta(z_\varepsilon)}|x|^{2\gamma}{u_\e}_{y_hy_h}{u_\e}_{y_j}
  +|x|^{2\gamma}{u_\e}_{y_h}{u_\e}_{y_jy_h}dz
  \\\leq&\int_{\widetilde{B}_\delta(z_\varepsilon)}|x|^{4\gamma}{u_\e}_{y_hy_h}^2+{u_\e}_{y_j}^2
  +|x|^{4\gamma}{u_\e}_{y_jy_h}^2+{u_\e}_{y_h}^2dz
  \\\leq&\widetilde{C}\int_{\widetilde{B}_\delta(z_\varepsilon)}e^{-\frac{\theta {d(z,z_\e)}}{\varepsilon}}dz.
\end{aligned}
\end{equation}
From \eqref{yph18}-\eqref{yph20}, we have the estimate
\begin{equation}\label{yph22}
  J_4=\int_{\partial\widetilde{B}_\delta(z_\e)}\left(\frac{\partial {u_\e}}{\partial \nu_x}+|x|^{2\gamma}\frac{\partial {u_\e}}{\partial \nu_y}\right)\frac{\partial {u_\e}}{\partial y_j}dS
  \leq\widetilde{C}\int_{\widetilde{B}_\delta(z_\varepsilon)}e^{-\frac{\theta {d(z,z_\e)}}{\varepsilon}}dz.
\end{equation}
From what we discuss above, we have proved \eqref{yph23}.
\end{proof}

\begin{proof}[\textbf{Proof of Corollary \ref{th1}:}]
Take $D=\widetilde{B}_\delta(z_\varepsilon)$, it follows from Lemma \ref{1010} that \eqref{.ma50} and \eqref{.ma60} are equivalent to
\begin{equation}\label{..10}
  \int_{\widetilde{B}_\delta(z_\varepsilon)}\frac{\partial a(z)}{\partial x_i}u_\varepsilon^2dz=O\left(e^{-\frac{\theta\delta}{\varepsilon}}+\int_{\widetilde{B}_\delta(z_\varepsilon)}e^{-\frac{\theta {d(z,z_\e)}}{\varepsilon}}dz\right),
\end{equation}

\begin{equation}\label{..11}
  \varepsilon^2\int_{\widetilde{B}_\delta(z_\varepsilon)}|\nabla_y u_\varepsilon|^2|x|^{2(\gamma-1)}x_idz=O\left(e^{-\frac{\theta\delta}{\varepsilon}}+\int_{\widetilde{B}_\delta(z_\varepsilon)}e^{-\frac{\theta {d(z,z_\e)}}{\varepsilon}}dz\right).
\end{equation}

\begin{equation}\label{..15}
 \int_{\widetilde{B}_\delta(z_\varepsilon)}\frac{\partial a(z)}{\partial y_j}u_\varepsilon^2dz=O\left(e^{-\frac{\theta\delta}{\varepsilon}}+\int_{\widetilde{B}_\delta(z_\varepsilon)}e^{-\frac{\theta {d(z,z_\e)}}{\varepsilon}}dz\right).
\end{equation}

On the other hand, from \eqref{..10}, let
\begin{equation*}
\begin{aligned}
  &b(\widetilde{z})=u_\varepsilon(\varepsilon\widetilde{x}+x_\varepsilon,\varepsilon^{\gamma+1}\widetilde{y}+y_\varepsilon),\;\;x=\varepsilon\widetilde{x}+x_\varepsilon,\;\;
  y=\varepsilon^{\gamma+1}\widetilde{y}+y_\varepsilon,
  \\&\widetilde{u_\varepsilon}(z)=u_\varepsilon(\varepsilon x+x_\varepsilon,\varepsilon^{\gamma+1}y+y_\varepsilon),
\end{aligned}
\end{equation*}
we can prove that
\begin{equation}\label{..12}
\begin{aligned}
&\int_{\widetilde{B}_\delta(z_\varepsilon)}\frac{\partial a(z)}{\partial x_i}u_\varepsilon^2dz
\\=&\int_{\widetilde{B}_\frac{\delta}{\varepsilon}(z_\varepsilon)}\frac{\partial a(\varepsilon\widetilde{x}+x_\varepsilon,\varepsilon^{\gamma+1}\widetilde{y}+y_\varepsilon)}{\widetilde{x_i}}\frac{1}{\varepsilon}\left(b(\widetilde{z})\right)^{2}
\varepsilon^{N+(\gamma+1)l}d\widetilde{z}
\\=&\varepsilon^{N+(\gamma+1)l-1}\int_{\widetilde{B}_\frac{\delta}{\varepsilon}(z_\varepsilon)}\frac{\partial a(\varepsilon x+x_\varepsilon,\varepsilon^{\gamma+1}y+y_\varepsilon)}{x_i}\left(\widetilde{u_\varepsilon}(z)\right)^{2}dz
\\=&\varepsilon^{N+(\gamma+1)l-1}\left(\frac{\partial a(z)}{\partial x_i}\Big|_{z=z_\varepsilon}\int_{\widetilde{B}_\frac{\delta}{\varepsilon}(z_\varepsilon)}\left(\widetilde{u_\varepsilon}(z)\right)^{2}dz+o(1)\right),
\end{aligned}
\end{equation}
where $\displaystyle\int_{\widetilde{B}_\frac{\delta}{\varepsilon}(z_\varepsilon)}\left(\widetilde{u_\varepsilon}(z)\right)^{2}dz$ is positive.
By combining \eqref{..10} and \eqref{..12}, we want to balance the exponentially small quantity and algebraically small quantity, thus we obtain
\begin{equation*}
  \frac{\partial a(z)}{\partial x_i}\Big|_{z=z_\varepsilon}=o(1)~ \mbox{for} ~i=1,\ldots,N,
\end{equation*}
i.e.
\begin{equation}\label{..13}
  \nabla_x a(z_0)=0.
\end{equation}

From \eqref{..11}, we can prove that
\begin{equation}\label{..16}
\begin{aligned}
&\varepsilon^2\int_{\widetilde{B}_\delta(z_\varepsilon)}|\nabla_y u_\varepsilon|^2|x|^{2(\gamma-1)}x_idz
\\=&\varepsilon^2\int_{\widetilde{B}_{\frac{\delta}{\varepsilon}}(0)}\frac{1}{\varepsilon^{2(\gamma+1)}}|\nabla_{\widetilde{y}}b(\widetilde{z})|^2
|\varepsilon\widetilde{x}+x_\varepsilon|^{2(\gamma-1)}(\varepsilon\widetilde{x_i}+x_{\varepsilon,i})\varepsilon^{N+(\gamma+1)l}d\widetilde{z}
\\=&\varepsilon^{N+(\gamma+1)(l-2)+2}\int_{\widetilde{B}_{\frac{\delta}{\varepsilon}}(0)}|\nabla_y\widetilde{u_\varepsilon}(z)|^2
|\varepsilon x+x_\varepsilon|^{2(\gamma-1)}(\varepsilon x_i+x_{\varepsilon,i})dz
\\=&\varepsilon^{N+(\gamma+1)(l-2)+2}\left(\int_{\widetilde{B}_{\frac{\delta}{\varepsilon}}(0)}|\nabla_y\widetilde{u_\varepsilon}(z)|^2
|\varepsilon x+x_\varepsilon|^{2(\gamma-1)}x_{\varepsilon,i}dz+o(1)\right),
\end{aligned}
\end{equation}
where $|\nabla_y\widetilde{u_\varepsilon}(z)|^2\geq0$.
By combining \eqref{..11} and \eqref{..16}, we want to balance the exponentially small quantity and algebraically small quantity, thus we obtain
\begin{equation*}
  x_{\varepsilon,i}=o(1)~ \mbox{for} ~i=1,\ldots,N,
\end{equation*}
i.e.
\begin{equation}\label{..17}
  x_0=0.
\end{equation}

Finally from \eqref{..15}, we can prove that
\begin{equation}\label{..18}
\begin{aligned}
&\int_{\widetilde{B}_\delta(z_\varepsilon)}\frac{\partial a(z)}{\partial y_j}u_\varepsilon^2dz
\\=&\int_{\widetilde{B}_\frac{\delta}{\varepsilon}(z_\varepsilon)}\frac{\partial a(\varepsilon\widetilde{x}+x_\varepsilon,\varepsilon^{\gamma+1}\widetilde{y}+y_\varepsilon)}{\widetilde{y_j}}\frac{1}{\varepsilon^{\gamma+1}}\left(b(\widetilde{z})\right)^{2}
\varepsilon^{N+(\gamma+1)l}d\widetilde{z}
\\=&\varepsilon^{N+(\gamma+1)(l-1)}\int_{\widetilde{B}_\frac{\delta}{\varepsilon}(z_\varepsilon)}\frac{\partial a(\varepsilon x+x_\varepsilon,\varepsilon^{\gamma+1}y+y_\varepsilon)}{y_j}\left(\widetilde{u_\varepsilon}(z)\right)^{2}dz
\\=&\varepsilon^{N+(\gamma+1)(l-1)}\left(\frac{\partial a(z)}{\partial y_j}\Big|_{z=z_\varepsilon}\int_{\widetilde{B}_\frac{\delta}{\varepsilon}(z_\varepsilon)}\left(\widetilde{u_\varepsilon}(z)\right)^{2}dz+o(1)\right),
\end{aligned}
\end{equation}
where $\displaystyle\int_{\widetilde{B}_\frac{\delta}{\varepsilon}(z_\varepsilon)}\left(\widetilde{u_\varepsilon}(z)\right)^{2}dz$ is positive.
By combining \eqref{..15} and \eqref{..18}, we want to balance the exponentially small quantity and algebraically small quantity, thus we obtain
\begin{equation*}
  \frac{\partial a(z)}{\partial y_j}\Big|_{z=z_\varepsilon}=o(1)~ \mbox{for} ~j=1,\ldots,l,
\end{equation*}
i.e.
\begin{equation}\label{..19}
  \nabla_y a(z_0)=0.
\end{equation}

According to \eqref{..13}, \eqref{..17} and \eqref{..19}, the concentration point $z_0$ for solution of \eqref{p10} satisfies $\nabla a(z_0)=0$ and $x_0=0$.
\end{proof}

\begin{Rem}
In Corollary \ref{th1}, we know that the concentration point $z_0=(0,y_0)$ is the critical point of $a(z)$ if $a(z)$ has critical points in this form. So hereafter, we also assume that the explosion point is in the form of $z_\varepsilon=(0,y_\varepsilon)$ for convenience of discussion.
\end{Rem}

\section{Pohozaev identities generated from scaling}

In this section, we prove Pohozaev identity generated from scaling of \eqref{ijk4} in Theorem \ref{300}. Then we study the nonexistence of nontrivial solutions for the Dirichlet problem \eqref{.ma9} under some assumptions by such identity in Corollary \ref{th1.1}.

\begin{proof}[\textbf{Proof of Theorem \ref{300}:}]
According to our notation in Theorem \ref{ma1}, we have \eqref{ma6} and \eqref{ma13}. Note that
\begin{equation}\label{ma19}
\begin{cases}
\big(x_iF(z,u)\big)_{x_i}=F(z,u)+x_i\big(\frac{\partial F(z,u)}{\partial x_i}+f(z,u)\frac{\partial u}{\partial x_i}\big),\\[3mm]
\big(y_jF(z,u)\big)_{y_j}=F(z,u)+y_j\big(\frac{\partial F(z,u)}{\partial y_j}+f(z,u)\frac{\partial u}{\partial y_j}\big),
\end{cases}
\end{equation}
for $i=1,\cdots,N$, and $j=1,\cdots,l$, in which $(\cdot)_{x_i}$ denotes derivation for each $x_i$ and $\frac{\partial(\cdot)}{\partial x_i}$ denotes derivation for $x_i$ in first variable.\\

From \eqref{ma19} and by integration by parts , we have
\begin{equation}\label{ma20}
\begin{aligned}
    \int_\Omega F(z,u)dz
    =&\int_\Omega\big(x_iF(z,u)\big)_{x_i}-x_i\frac{\partial F(z,u)}{\partial x_i}-x_if(z,u)\frac{\partial u}{\partial x_i}dz
  \\\overset{I}=&\int_{\partial\Omega}x_iF(z,u)\nu_x^idS-\int_\Omega x_i\frac{\partial F(z,u)}{\partial x_i}dz-\int_\Omega x_if(z,u)\frac{\partial u}{\partial x_i}dz,
\end{aligned}
\end{equation}
and
\begin{equation}\label{ma21}
\begin{aligned}
    \int_\Omega F(z,u)dz
    =&\int_\Omega\big(y_jF(z,u)\big)_{y_j}-y_j\frac{\partial F(z,u)}{\partial y_j}-y_jf(z,u)\frac{\partial u}{\partial y_j}dz
  \\\overset{I}=&\int_{\partial\Omega}y_jF(z,u)\nu_y^jdS-\int_\Omega y_j\frac{\partial F(z,u)}{\partial y_j}dz-\int_\Omega y_jf(z,u)\frac{\partial u}{\partial y_j}dz.
\end{aligned}
\end{equation}
Sum \eqref{ma20} from $1$ to $N$ and sum \eqref{ma21} from $1$ to $l$ respectively, we obtain
\begin{equation}\label{ma22}
\begin{aligned}
  N\int_\Omega F(z,u)dz&=\int_{\partial \Omega}F(z,u)(x\cdot\nu_x)dS-\int_\Omega x\cdot\nabla_xF(z,u)dz-\int_\Omega (x\cdot\nabla_xu)f(z,u)dz
\\&=-\int_\Omega x\cdot\nabla_xF(z,u)dz-\int_\Omega(x\cdot\nabla_xu)f(z,u)dz,
\end{aligned}
\end{equation}
and
\begin{equation}\label{ma23}
\begin{aligned}
  l\int_\Omega F(z,u)dz&=\int_{\partial \Omega}F(z,u)(y\cdot\nu_y)dS-\int_\Omega y\cdot\nabla_yF(z,u)dz-\int_\Omega (y\cdot\nabla_yu)f(z,u)dz
  \\&=-\int_\Omega y\cdot\nabla_yF(z,u)dz-\int_\Omega(y\cdot\nabla_yu)f(z,u)dz,
\end{aligned}
\end{equation}
since $F(z,u)=0$ on $\partial\Omega$.

Add \eqref{ma22} and \eqref{ma23}, we have
\begin{equation}\label{ma24}
  (N+l)\int_\Omega F(z,u)dz=-\int_\Omega z\cdot\nabla_zF(z,u)dz-\int_\Omega(z\cdot\nabla_zu)f(z,u)dz,
\end{equation}
where $\nabla_zF(z,u)$ represents the gradient of $F$ with respect to $z$, and this gradient has nothing to do with $u$.

Let's calculate the last term of RHS of \eqref{ma24}, since other terms will not be changed.
\begin{equation}\label{ma25}
\begin{aligned}
  &-\int_\Omega(z\cdot\nabla_zu)f(z,u)dz
  \\\overset{\eqref{ijk4}}=&\int_\Omega(z\cdot\nabla_zu)\Delta_\gamma u dz
  \\=&\int_\Omega(x\cdot\nabla_xu+y\cdot\nabla_yu)(\Delta_xu+|x|^{2\gamma}\Delta_yu)dz
  \\=&\int_\Omega\Delta_xu(x\cdot\nabla_xu)dz
    +\int_\Omega\Delta_yu(|x|^{2\gamma} x\cdot\nabla_xu)dz
  \\+&\int_\Omega\Delta_xu(y\cdot\nabla_yu)dz
    +\int_\Omega\Delta_yu(|x|^{2\gamma}y\cdot\nabla_yu)dz
    \\:=&I_1+I_2+I_3+I_4.
\end{aligned}
\end{equation}
Next we will apply Green first formula and integration by parts to calculate four terms in \eqref{ma25} one by one for the purpose of our aim.\\
First term $I_1$:
\begin{equation}\label{ma26}
\begin{aligned}
  &\int_\Omega\Delta_xu(x\cdot\nabla_xu)dz
  \\\overset{G}=&\int_{\partial \Omega}(x\cdot\nabla_xu)\frac{\partial u}{\partial \nu_x}dS-\int_\Omega\nabla_x(x\cdot\nabla_xu)\cdot\nabla_xudz
  \\\overset{I}=&\int_{\partial \Omega}|\nabla_xu|^2(x\cdot\nu_x)dS-(1-\frac{N}{2})\int_\Omega|\nabla_xu|^2dz-\frac{1}{2}\int_{\partial \Omega}|\nabla_xu|^2(x\cdot\nu_x)dS
  \\=&\frac{1}{2}\int_{\partial \Omega}|\nabla_xu|^2(x\cdot\nu_x)dS+(\frac{N}{2}-1)\int_\Omega|\nabla_xu|^2dz.
\end{aligned}
\end{equation}
Second term $I_2$:
\begin{equation}\label{ma27}
\begin{aligned}
  &\int_\Omega\Delta_yu(|x|^{2\gamma}x\cdot\nabla_xu)dz
\\\overset{G}=&\int_{\partial \Omega}(x\cdot\nabla_xu)|x|^{2\gamma}\frac{\partial u}{\partial\nu_y}dS-\int_\Omega\nabla_y\big((x\cdot\nabla_xu)|x|^{2\gamma}\big)\cdot\nabla_yudz
\\=&\int_{\partial \Omega}(x\cdot\nabla_xu)|x|^{2\gamma}(\nabla_yu\cdot\nu_y)dS-\frac{1}{2}\int_\Omega|x|^{2\gamma}\big(x\cdot\nabla_x(|\nabla_yu|^2)\big)dz
\\\overset{I}=&\int_{\partial \Omega}|x|^{2\gamma}(x\cdot\nabla_xu)(\nabla_yu\cdot\nu_y)dS
     +(\frac{N}{2}+\gamma)\int_\Omega|x|^{2\gamma}|\nabla_yu|^2dz
     -\frac{1}{2}\int_{\partial \Omega}|x|^{2\gamma}|\nabla_yu|^2(x\cdot\nu_x)dS.
\end{aligned}
\end{equation}
Third term $I_3$:
\begin{equation}\label{ma28}
\begin{aligned}
  &\int_\Omega\Delta_xu(y\cdot\nabla_yu)dz
  \\\overset{G}=&\int_{\partial \Omega}(y\cdot\nabla_yu)\frac{\partial u}{\partial\nu_x}dS-\int_\Omega\nabla_x(y\cdot\nabla_yu)\cdot\nabla_xudz
   \\=&\int_{\partial \Omega}(y\cdot\nabla_yu)(\nabla_xu\cdot\nu_x)dS-\frac{1}{2}\int_\Omega y\cdot\nabla_y(|\nabla_xu|^2)dz
   \\\overset{I}=&\int_{\partial \Omega}(y\cdot\nabla_yu)(\nabla_xu\cdot\nu_x)dS+\frac{l}{2}\int_\Omega|\nabla_xu|^2dz-\frac{1}{2}\int_{\partial \Omega}|\nabla_xu|^2(y\cdot\nu_y)dS.
\end{aligned}
\end{equation}
Fourth term $I_4$:
\begin{equation}\label{ma29}
\begin{aligned}
  &\int_\Omega\Delta_yu(|x|^{2\gamma}y\cdot\nabla_yu)dz
  \\\overset{G}=&\int_{\partial \Omega}(y\cdot\nabla_yu)|x|^{2\gamma}\frac{\partial u}{\partial\nu_y}dS-\int_\Omega\nabla_y\big((y\cdot\nabla_yu)|x|^{2\gamma}\big)\cdot\nabla_yudz
  \\=&\int_{\partial \Omega}(y\cdot\nabla_yu)|x|^{2\gamma}(\nabla_yu\cdot\nu_y)dS-\int_\Omega|x|^{2\gamma}|\nabla_yu|^2dz-\frac{1}{2}\int_\Omega|x|^{2\gamma}\big(y\cdot\nabla_y(|\nabla_yu|^2)\big)dz
   \\\overset{I}=&\int_{\partial \Omega}(y\cdot\nu_y)|x|^{2\gamma}|\nabla_yu|^2dS+(\frac{l}{2}-1)\int_\Omega|x|^{2\gamma}|\nabla_yu|^2dz-\frac{1}{2}\int_{\partial \Omega}(y\cdot\nu_y)|x|^{2\gamma}|\nabla_yu|^2dS
   \\=&\frac{1}{2}\int_{\partial \Omega}|x|^{2\gamma}|\nabla_yu|^2(y\cdot\nu_y)dS+(\frac{l}{2}-1)\int_\Omega|x|^{2\gamma}|\nabla_yu|^2dz.
\end{aligned}
\end{equation}
From \eqref{ma24}-\eqref{ma29}, we deduce that
\begin{equation}\label{ma30}
\begin{aligned}
   &(N+l)\int_\Omega F(z,u)dz
\\=&-\int_\Omega z\cdot\nabla_zF(z,u)dz
\\+&(\frac{N}{2}-1)\int_\Omega|\nabla_xu|^2dz+\frac{1}{2}\int_{\partial \Omega}|\nabla_xu|^2(x\cdot\nu_x)dS
\\+&\int_{\partial \Omega}|x|^{2\gamma}(x\cdot\nabla_xu)(\nabla_yu\cdot\nu_y)dS
     +(\frac{N}{2}+\gamma)\int_\Omega|x|^{2\gamma}|\nabla_yu|^2dz
     -\frac{1}{2}\int_{\partial \Omega}|x|^{2\gamma}|\nabla_yu|^2(x\cdot\nu_x)dS
\\+&\int_{\partial \Omega}(y\cdot\nabla_yu)(\nabla_xu\cdot\nu_x)dS+\frac{l}{2}\int_\Omega|\nabla_xu|^2dz-\frac{1}{2}\int_{\partial \Omega}|\nabla_xu|^2(y\cdot\nu_y)dS
\\+&\frac{1}{2}\int_{\partial \Omega}|x|^{2\gamma}|\nabla_yu|^2(y\cdot\nu_y)dS+(\frac{l}{2}-1)\int_\Omega|x|^{2\gamma}|\nabla_yu|^2dz.
\end{aligned}
\end{equation}
Note that $\nabla_xu=\frac{\partial u}{\partial\nu}\nu_x$, $\nabla_yu=\frac{\partial u}{\partial\nu}\nu_y$, since $u=0$ on $\partial\Omega$. Thus we obtain
\begin{equation}\label{ma31}
\begin{aligned}
   &(N+l)\int_\Omega F(z,u)dz
\\=&-\int_\Omega z\cdot\nabla_zF(z,u)dz
\\+&(\frac{N+l}{2}-1)\int_\Omega|\nabla_xu|^2dz+(\frac{N+l}{2}+\gamma-1)\int_\Omega|x|^{2\gamma}|\nabla_yu|^2dz
\\+&\frac{1}{2}\int_{\partial\Omega}\left(\frac{\partial u}{\partial\nu}\right)^2(\nu_x^2+|x|^{2\gamma}\nu_y^2)(x\cdot\nu_x+y\cdot\nu_y)dS
\\=&-\int_\Omega z\cdot\nabla_zF(z,u)dz
+(\frac{N+l}{2}-1)\int_\Omega|\nabla_\gamma u|^2dz+\gamma\int_\Omega|x|^{2\gamma}|\nabla_yu|^2dz
\\+&\frac{1}{2}\int_{\partial\Omega}\left(\frac{\partial u}{\partial\nu}\right)^2(\nu_x^2+|x|^{2\gamma}\nu_y^2)(x\cdot\nu_x+y\cdot\nu_y)dS.
\end{aligned}
\end{equation}
On the other hand by Green first formula,
\begin{equation}\label{ma32}
   -\int_\Omega f(z,u)udz
=-\int_\Omega|\nabla_\gamma u|^2dz,
\end{equation}
since $u=0$ on $\partial\Omega$.

Our result \eqref{.ma7} follows from \eqref{ma31} and \eqref{ma32}.
\end{proof}

Now we give an application of Pohozaev identity \eqref{.ma7}.
Consider the problem \eqref{.ma9}, we prove that there is only trivial solution for this problem under certain assumptions.

\begin{proof}[\textbf{Proof of Corollary \ref{th1.1}:}]
We apply Theorem \ref{300} directly. For \eqref{.ma9}, $f(z,u)=u^{\frac{N+l+2}{N+l-2}}$,~$F(z,u)=\frac{N+l-2}{2(N+l)}u^{\frac{2(N+l)}{N+l-2}}$.
So we have
\begin{equation}\label{zs2}
  (N+l)\int_\Omega F(z,u)dz=\left(\frac{N+l}{2}-1\right)\int_\Omega f(z,u)udz.
\end{equation}
Since $\nabla_zF(z,u)=0$, in \eqref{.ma7},
\begin{equation}\label{zs1}
  \int_\Omega z\cdot\nabla_z F(z,u)dz=0.
\end{equation}
Therefore, we find that \eqref{.ma7} transforms
\begin{equation}\label{.ma10}
\begin{aligned}
\gamma\int_\Omega|x|^{2\gamma}|\nabla_yu|^2dz
+\frac{1}{2}\int_{\partial\Omega}\left(\frac{\partial u}{\partial\nu}\right)^2(\nu_x^2+|x|^{2\gamma}\nu_y^2)(x\cdot\nu_x+y\cdot\nu_y)dS=0.
\end{aligned}
\end{equation}
If~$u$~is a nontrival solution of \eqref{.ma9},then $u$ satisfies \eqref{.ma10}.
According our assumption, the LHS of \eqref{.ma10} is positive, but the RHS of \eqref{.ma10} is equal to $0$.
This is impossible, thus $u=0$.\\
\end{proof}

\section{Kelvin transformation}

In this section, we prove the change of Grushin operator by Kelvin transformation in Theorem \ref{th2}. Then we estimate the decay rate of solution at infinity for the problem \eqref{zzzzzzzzzzzzzzz} in Corollary \ref{th3}.

\begin{proof}[\textbf{Proof of Theorem \ref{th2}:}]
We consider Grushin operator in the form of
\begin{equation}\label{1}
  \Delta_\gamma u=\Delta_xu+|x|^{2\gamma}\Delta_yu,
\end{equation}
where $z=(x,y)\in\R^{N+l}$.

We recall new distance between $0$ and $z$ on $\R^{N+l}$ as defined in Definition \ref{zrf1}:
\begin{equation}\label{2}
  d(z,0)=\left(\frac{1}{(1+\gamma)^2}|x|^{2+2\gamma}+|y|^2\right)^{\frac{1}{2+2\gamma}}.
\end{equation}
There is an any point $\widetilde{z}=(\widetilde{x},\widetilde{y})\in\R^{N+l}$.
Let
\begin{equation}\label{3}
  \widetilde{x}=\frac{x}{d^2(z,0)},\;\;\;\;\widetilde{y}=\frac{y}{d^{2+2\gamma}(z,0)}.
\end{equation}
Then we have
\begin{equation}\label{4}
  d(\widetilde{z},0)=\frac{1}{d(z,0)},
\end{equation}
and
\begin{equation}\label{5}
  x=\frac{\widetilde{x}}{d^2(\widetilde{z},0)},\;\;\;\;y=\frac{\widetilde{y}}{d^{2+2\gamma}(\widetilde{z},0)}.
\end{equation}
There are two functions: $u=u(z)$, $v=v(z)$. Note that
\begin{equation}\label{6}
\begin{aligned}
  \Delta_\gamma(uv)
  =&\Delta_x(uv)+|x|^{2\gamma}\Delta_y(uv)
\\=&v\Delta_xu+u\Delta_xv+2\nabla_xu\cdot\nabla_xv
\\+&|x|^{2\gamma}\left(v\Delta_yu+u\Delta_yv+2\nabla_yu\cdot\nabla_yv\right)
\\=&v\Delta_\gamma u+u\Delta_\gamma v+2\left(\nabla_xu\cdot\nabla_xv+|x|^{2\gamma}\nabla_yu\cdot\nabla_yv\right).
\end{aligned}
\end{equation}
We put
\begin{equation}\label{jbbb}
  w(\widetilde{z})=d^b(z(\widetilde{z}),0)u(z(\widetilde{z})),
\end{equation}
where $b>0$ is an undetermined coefficient and denote by $\Delta_\gamma$, $\Delta_\gamma^{(\widetilde{z})}$ the Grushin operator acting in $z$ and $\widetilde{z}$ variable. According to \eqref{6}, we have
\begin{equation}\label{7}
\begin{aligned}
  \Delta_\gamma^{(\widetilde{z})}w(\widetilde{z})
  =&d^b(z(\widetilde{z}),0)\Delta_\gamma^{(\widetilde{z})}u(z(\widetilde{z}))
  +u(z(\widetilde{z}))\Delta_\gamma^{(\widetilde{z})}d^b(z(\widetilde{z}),0)
  \\+&2\left(\nabla_{\widetilde{x}}u(z(\widetilde{z}))\cdot\nabla_{\widetilde{x}}d^b(z(\widetilde{z}),0)
  +|\widetilde{x}|^{2\gamma}
  \nabla_{\widetilde{y}}u(z(\widetilde{z}))\cdot\nabla_{\widetilde{y}}d^b(z(\widetilde{z}),0)\right).
\end{aligned}
\end{equation}
In \eqref{7}, if we take $b=N+(1+\gamma)l-2$, then
\begin{equation}\label{8}
  u(z(\widetilde{z}))\Delta_\gamma^{(\widetilde{z})}d^b(z(\widetilde{z}),0)=0.
\end{equation}

Indeed, note that $d^b(z(\widetilde{z}),0)=d^{-b}(\widetilde{z},0)$.
Let's take some derivatives for preparation:
\begin{equation}\label{9}
  \frac{\partial d(\widetilde{z},0)}{\partial\widetilde{x_i}}=\frac{1}{(1+\gamma)^2}\widetilde{x_i}
  |\widetilde{x}|^{2\gamma}d^{-1-2\gamma}(\widetilde{z},0),
\end{equation}
\begin{equation}\label{10}
  \frac{\partial d(\widetilde{z},0)}{\partial\widetilde{y_j}}=\frac{1}{1+\gamma}\widetilde{y_j}d^{-1-2\gamma}(\widetilde{z},0),
\end{equation}
\begin{equation}\label{11}
  \frac{\partial d^{-b}(\widetilde{z},0)}{\partial\widetilde{x_i}}
  =-b\frac{1}{(1+\gamma)^2}\widetilde{x_i}
  |\widetilde{x}|^{2\gamma}d^{-b-2-2\gamma}(\widetilde{z},0),
\end{equation}
\begin{equation}\label{12}
   \frac{\partial d^{-b}(\widetilde{z},0)}{\partial\widetilde{y_j}}
   =-b\frac{1}{1+\gamma}\widetilde{y_j}d^{-b-2-2\gamma}(\widetilde{z},0),
\end{equation}
\begin{equation}\label{13}
\begin{aligned}
  \frac{\partial^2d^{-b}(\widetilde{z})}{\partial\widetilde{x_i}^2}
  =-&b\frac{1}{(1+\gamma)^2}|\widetilde{x}|^{2\gamma}d^{-b-2-2\gamma}(\widetilde{z},0)
\\-&2b\frac{\gamma}{(1+\gamma)^2}\widetilde{x_i}^2|\widetilde{x}|^{2\gamma-2}d^{-b-2-2\gamma}(\widetilde{z},0)
\\+&b(b+2+2\gamma)\frac{1}{(1+\gamma)^4}\widetilde{x_i}^2|\widetilde{x}|^{4\gamma}d^{-b-4-4\gamma}(\widetilde{z},0),
\end{aligned}
\end{equation}
\begin{equation}\label{14}
\begin{aligned}
  \frac{\partial^2d^{-b}(\widetilde{z})}{\partial\widetilde{y_j}^2}
  =-&b\frac{1}{1+\gamma}d^{-b-2-2\gamma}(\widetilde{z},0)
\\+&b(b+2+2\gamma)\frac{1}{(1+\gamma)^2}\widetilde{y_j}^2d^{-b-4-4\gamma}(\widetilde{z},0).
\end{aligned}
\end{equation}
Sum \eqref{13} from $i=1$ to $N$, we can obtain
\begin{equation}\label{15}
\begin{aligned}
  \Delta_{\widetilde{x}}d^{-b}(\widetilde{z},0)
  =-&bN\frac{1}{(1+\gamma)^2}|\widetilde{x}|^{2\gamma}d^{-b-2-2\gamma}(\widetilde{z},0)
  \\-&2b\frac{\gamma}{(1+\gamma)^2}|\widetilde{x}|^{2\gamma}d^{-b-2-2\gamma}(\widetilde{z},0)
  \\+&b(b+2+2\gamma)\frac{1}{(1+\gamma)^4}|\widetilde{x}|^{4\gamma+2}d^{-b-4-4\gamma}(\widetilde{z},0).
\end{aligned}
\end{equation}
Sum \eqref{14} from $j=1$ to $l$, we can obtain
\begin{equation}\label{16}
\begin{aligned}
  \Delta_{\widetilde{y}}d^{-b}(\widetilde{z},0)
  =-&bl\frac{1}{1+\gamma}d^{-b-2-2\gamma}(\widetilde{z},0)
  \\+&b(b+2+2\gamma)\frac{1}{(1+\gamma)^2}|\widetilde{y}|^2d^{-b-4-4\gamma}(\widetilde{z},0).
\end{aligned}
\end{equation}
By \eqref{15} plus \eqref{16} times $|\widetilde{x}|^{2\gamma}$, we can obtain
\begin{equation}\label{17}
  \Delta_\gamma^{(\widetilde{z})}d^{-b}(\widetilde{z},0)
  =b(b-N-(1+\gamma)l+2)\frac{1}{(1+\gamma)^2}|\widetilde{x}|^{2\gamma}d^{-b-2-2\gamma}(\widetilde{z},0).
\end{equation}
On the basis of our assumption for $b$, we take $b=N+(1+\gamma)l-2$, then the term \eqref{8} is equal to $0$.

Next we calculate the term
\begin{equation}\label{18}
\begin{aligned}
  &d^b(z(\widetilde{z}),0)\Delta_\gamma^{(\widetilde{z})}u(z(\widetilde{z}))
  \\=&d^{-b}(\widetilde{z},0)\Delta_\gamma^{(\widetilde{z})}u(z(\widetilde{z})).
\end{aligned}
\end{equation}
According to the chain rule of derivative, we have
\begin{equation}\label{19}
  \frac{\partial u}{\partial\widetilde{x_k}}=\sum^N_{i=1}\frac{\partial u}{\partial x_i}\frac{\partial x_i}{\partial\widetilde{x_k}}+\sum^l_{j=1}\frac{\partial u}{\partial y_j}\frac{\partial y_j}{\partial\widetilde{x_k}}
  ,\;\;\;\;k=1,\cdots,N,
\end{equation}
\begin{equation}\label{20}
  \frac{\partial u}{\partial\widetilde{y_h}}=\sum^N_{i=1}\frac{\partial u}{\partial x_i}\frac{\partial x_i}{\partial\widetilde{y_h}}+\sum^l_{j=1}\frac{\partial u}{\partial y_j}\frac{\partial y_j}{\partial\widetilde{y_h}}
  ,\;\;\;\;h=1,\cdots,l.
\end{equation}
Thus, we have
\begin{equation}\label{21}
\begin{aligned}
  \frac{\partial^2u}{\partial\widetilde{x_k}^2}=&\sum^N_{i=1}\frac{\partial^2u}{\partial x_i^2}\big(\frac{\partial x_i}{\partial\widetilde{x_k}}\big)^2+\sum^l_{j=1}\frac{\partial^2u}{\partial y_j^2}\big(\frac{\partial y_j}{\partial\widetilde{x_k}}\big)^2
  \\+&\sum^N_{i=1}\frac{\partial u}{\partial x_i}\frac{\partial^2x_i}{\partial\widetilde{x_k}^2}+\sum^l_{j=1}\frac{\partial u}{\partial y_j}\frac{\partial^2y_j}{\partial\widetilde{x_k}^2}
  \\+&\left(\sum^N_{\substack{i,i'=1\\i\neq i'}}\frac{\partial^2u}{\partial x_i\partial x_{i'}}\frac{\partial x_i}{\partial\widetilde{x_k}}\frac{\partial x_{i'}}{\partial\widetilde{x_k}}
  +\sum^l_{\substack{j,j'=1\\j\neq j'}}\frac{\partial^2u}{\partial y_j\partial y_{j'}}\frac{\partial y_j}{\partial\widetilde{x_k}}\frac{\partial y_{j'}}{\partial\widetilde{x_k}}
  +\sum^{N,l}_{i,j=1}\frac{\partial^2u}{\partial x_i\partial y_j}\frac{\partial x_i}{\partial\widetilde{x_k}}\frac{\partial y_j}{\partial\widetilde{x_k}}\right),
\end{aligned}
\end{equation}
for $k=1,\cdots,N$ and
\begin{equation}\label{22}
\begin{aligned}
  \frac{\partial^2u}{\partial\widetilde{y_h}^2}=&\sum^N_{i=1}\frac{\partial^2u}{\partial x_i^2}\big(\frac{\partial x_i}{\partial\widetilde{y_h}}\big)^2+\sum^l_{j=1}\frac{\partial^2u}{\partial y_j^2}\big(\frac{\partial y_j}{\partial\widetilde{y_h}}\big)^2
  \\+&\sum^N_{i=1}\frac{\partial u}{\partial x_i}\frac{\partial^2x_i}{\partial\widetilde{y_h}^2}+\sum^l_{j=1}\frac{\partial u}{\partial y_j}\frac{\partial^2y_j}{\partial\widetilde{y_h}^2}
  \\+&\left(\sum^N_{\substack{i,i'=1\\i\neq i'}}\frac{\partial^2u}{\partial x_i\partial x_{i'}}\frac{\partial x_i}{\partial\widetilde{y_h}}\frac{\partial x_{i'}}{\partial\widetilde{y_h}}
  +\sum^l_{\substack{j,j'=1\\j\neq j'}}\frac{\partial^2u}{\partial y_j\partial y_{j'}}\frac{\partial y_j}{\partial\widetilde{y_h}}\frac{\partial y_{j'}}{\partial\widetilde{y_h}}
  +\sum^{N,l}_{i,j=1}\frac{\partial^2u}{\partial x_i\partial y_j}\frac{\partial x_i}{\partial\widetilde{y_h}}\frac{\partial y_j}{\partial\widetilde{y_h}}\right),
\end{aligned}
\end{equation}
for $h=1,\cdots,l$.\\
Note that $\Delta_\gamma^{(\widetilde{z})}u(z(\widetilde{z}))=\Delta_{\widetilde{x}}u(z(\widetilde{z}))
+|\widetilde{x}|^{2\gamma}\Delta_{\widetilde{y}}u(z(\widetilde{z}))$. So by \eqref{21} and \eqref{22} we have
\begin{equation*}
\begin{aligned}
  &\Delta_\gamma^{(\widetilde{z})}u(z(\widetilde{z}))
  \\=&\sum^N_{i=1}\frac{\partial^2u}{\partial x_i^2}\left[\sum^N_{k=1}\big(\frac{\partial x_i}{\partial\widetilde{x_k}}\big)^2+|\widetilde{x}|^{2\gamma}\sum^l_{h=1}\big(\frac{\partial x_i}{\partial\widetilde{y_h}}\big)^2\right]
  \\+&\sum^l_{j=1}\frac{\partial^2u}{\partial y_j^2}\left[\sum^N_{k=1}\big(\frac{\partial y_j}{\partial\widetilde{x_k}}\big)^2+|\widetilde{x}|^{2\gamma}\sum^l_{h=1}\big(\frac{\partial y_j}{\partial\widetilde{y_h}}\big)^2\right]
  \\+&\sum^N_{i=1}\frac{\partial u}{\partial x_i}\left(\sum^N_{k=1}\frac{\partial^2x_i}{\partial\widetilde{x_k}^2}+|\widetilde{x}|^{2\gamma}
  \sum^l_{h=1}\frac{\partial^2x_i}{\partial\widetilde{y_h}^2}\right)
  \\+&\sum^l_{j=1}\frac{\partial u}{\partial y_j}\left(\sum^N_{k=1}\frac{\partial^2y_j}{\partial\widetilde{x_k}^2}+|\widetilde{x}|^{2\gamma}
  \sum^l_{h=1}\frac{\partial^2y_j}{\partial\widetilde{y_h}^2}\right)
\end{aligned}
\end{equation*}
\begin{equation}\label{23}
\begin{aligned}
  +&\sum^N_{\substack{i,i'=1\\i\neq i'}}\frac{\partial^2u}{\partial x_i\partial x_{i'}}\left(\sum^N_{k=1}\frac{\partial x_i}{\partial\widetilde{x_k}}\frac{\partial x_{i'}}{\partial\widetilde{x_k}}+|\widetilde{x}|^{2\gamma}\sum^l_{h=1}\frac{\partial x_i}{\partial\widetilde{y_h}}\frac{\partial x_{i'}}{\partial\widetilde{y_h}}\right)
  \\+&\sum^l_{\substack{j,j'=1\\j\neq j'}}\frac{\partial^2u}{\partial y_j\partial y_{j'}}\left(\sum^N_{k=1}\frac{\partial y_j}{\partial\widetilde{x_k}}\frac{\partial y_{j'}}{\partial\widetilde{x_k}}+|\widetilde{x}|^{2\gamma}\sum^l_{h=1}\frac{\partial y_j}{\partial\widetilde{y_h}}\frac{\partial y_{j'}}{\partial\widetilde{y_h}}\right)
  \\+&\sum^{N,l}_{i,j=1}\frac{\partial^2u}{\partial x_i\partial y_j}\left(\sum^N_{k=1}\frac{\partial x_i}{\partial\widetilde{x_k}}\frac{\partial y_j}{\partial\widetilde{x_k}}+|\widetilde{x}|^{2\gamma}\sum^l_{h=1}\frac{\partial x_i}{\partial\widetilde{y_h}}\frac{\partial y_j}{\partial\widetilde{y_h}}\right).
\end{aligned}
\end{equation}
We also take some derivatives for preparation as same as before.\\
According to \eqref{5}, i.e.
\begin{equation}\label{24}
  x_i=\widetilde{x_i}d^{-2}(\widetilde{z},0),\;\;\;\;y_j=\widetilde{y_j}d^{-2-2\gamma}(\widetilde{z},0),
\end{equation}
for $i=1,\cdots,N$, $J=1,\cdots,l$ and according to \eqref{9}, \eqref{10}, we have
\begin{flalign}\label{25}
\frac{\partial x_i}{\partial\widetilde{x_k}}=
  \begin{cases}
d^{-2}(\widetilde{z},0)-2\frac{1}{(1+\gamma)^2}\widetilde{x_i}^2|\widetilde{x}|^{2\gamma}d^{-4-2\gamma}(\widetilde{z},0),
    &~\mbox{when}~k=i,\\
 -2\frac{1}{(1+\gamma)^2}\widetilde{x_i}\widetilde{x_k}|\widetilde{x}|^{2\gamma}d^{-4-2\gamma}(\widetilde{z},0), & ~\mbox{when}~k\neq i.
  \end{cases}
\end{flalign}
\begin{equation}\label{26}
  \frac{\partial x_i}{\partial\widetilde{y_h}}=-2\frac{1}{1+\gamma}\widetilde{x_i}\widetilde{y_h}d^{-4-2\gamma}(\widetilde{z},0).
\end{equation}
\begin{equation}\label{27}
  \frac{\partial y_j}{\partial\widetilde{x_k}}=-2\frac{1}{1+\gamma}\widetilde{y_j}\widetilde{x_k}|\widetilde{x}|^{2\gamma}
  d^{-4-4\gamma}(\widetilde{z},0).
\end{equation}
\begin{flalign}\label{28}
\frac{\partial y_j}{\partial\widetilde{y_h}}=
  \begin{cases}
d^{-2-2\gamma}(\widetilde{z},0)-2\widetilde{y_j}^2d^{-4-4\gamma}(\widetilde{z},0),
    &~\mbox{when}~h=j,\\
 -2\widetilde{y_j}\widetilde{y_h}d^{-4-4\gamma}(\widetilde{z},0), & ~\mbox{when}~h\neq j.
  \end{cases}
\end{flalign}
Moreover, we need to take the second derivatives and the square of first derivatives from \eqref{25}-\eqref{28}.
\begin{equation}\label{29}
\begin{aligned}
  \frac{\partial^2x_i}{\partial\widetilde{x_i}^2}=
  &-6\frac{1}{(1+\gamma)^2}\widetilde{x_i}|\widetilde{x}|^{2\gamma}d^{-4-2\gamma}(\widetilde{z},0)
\\&-4\gamma\frac{1}{(1+\gamma)^2}\widetilde{x_i}^3|\widetilde{x}|^{2\gamma-2}d^{-4-2\gamma}(\widetilde{z},0)
\\&+4(2+\gamma)\frac{1}{(1+\gamma)^4}\widetilde{x_i}^3|\widetilde{x}|^{4\gamma}d^{-6-4\gamma}(\widetilde{z},0).
\end{aligned}
\end{equation}
\begin{equation}\label{30}
\begin{aligned}
  \frac{\partial^2x_i}{\partial\widetilde{x_k}^2}=
  &-2\frac{1}{(1+\gamma)^2}\widetilde{x_i}|\widetilde{x}|^{2\gamma}d^{-4-2\gamma}(\widetilde{z},0)
\\&-4\gamma\frac{1}{(1+\gamma)^2}\widetilde{x_i}\widetilde{x_k}^2|\widetilde{x}|^{2\gamma-2}d^{-4-2\gamma}(\widetilde{z},0)
\\&+4(2+\gamma)\frac{1}{(1+\gamma)^4}\widetilde{x_i}\widetilde{x_k}^2|\widetilde{x}|^{4\gamma}
d^{-6-4\gamma}(\widetilde{z},0), ~\mbox{when}~k\neq i.
\end{aligned}
\end{equation}
\begin{equation}\label{31}
  \frac{\partial^2x_i}{\partial\widetilde{y_h}^2}=-2\frac{1}{1+\gamma}\widetilde{x_i}d^{-4-2\gamma}(\widetilde{z},0)
  +4(2+\gamma)\frac{1}{(1+\gamma)^2}\widetilde{x_i}\widetilde{y_h}^2d^{-6-4\gamma}(\widetilde{z},0).
\end{equation}
\begin{equation}\label{32}
\begin{aligned}
  \frac{\partial^2y_j}{\partial\widetilde{x_k}^2}=
  &-2\frac{1}{1+\gamma}\widetilde{y_j}|\widetilde{x}|^{2\gamma}d^{-4-4\gamma}(\widetilde{z},0)
\\&-4\gamma\frac{1}{1+\gamma}\widetilde{y_j}\widetilde{x_k}^2|\widetilde{x}|^{2\gamma-2}d^{-4-4\gamma}(\widetilde{z},0)
\\&+8\frac{1}{(1+\gamma)^2}\widetilde{y_j}\widetilde{x_k}^2|\widetilde{x}|^{4\gamma}d^{-6-6\gamma}(\widetilde{z},0).
\end{aligned}
\end{equation}
\begin{equation}\label{33}
  \frac{\partial^2 y_j}{\partial\widetilde{y_j}^2}=-6\widetilde{y_j}d^{-4-4\gamma}(\widetilde{z},0)
  +8\widetilde{y_j}^3d^{-6-6\gamma}(\widetilde{z},0).
\end{equation}
\begin{equation}\label{34}
  \frac{\partial^2 y_j}{\partial\widetilde{y_h}^2}=-2\widetilde{y_j}d^{-4-4\gamma}(\widetilde{z},0)
  +8\widetilde{y_j}\widetilde{y_h}^2d^{-6-6\gamma}(\widetilde{z},0), ~\mbox{when}~h\neq j.
\end{equation}
\begin{flalign}\label{35}
\big(\frac{\partial x_i}{\partial\widetilde{x_k}}\big)^2=
  \begin{cases}
d^{-4}(\widetilde{z},0)+4\frac{1}{(1+\gamma)^4}\widetilde{x_i}^4|\widetilde{x}|^{4\gamma}d^{-8-4\gamma}(\widetilde{z},0)
-4\frac{1}{(1+\gamma)^2}\widetilde{x_i}^2|\widetilde{x}|^{2\gamma}d^{-6-2\gamma}(\widetilde{z},0),
    &k=i,\\
 4\frac{1}{(1+\gamma)^4}\widetilde{x_i}^2\widetilde{x_k}^2|\widetilde{x}|^{4\gamma}d^{-8-4\gamma}(\widetilde{z},0), & k\neq i.
  \end{cases}
\end{flalign}
\begin{equation}\label{36}
  \big(\frac{\partial x_i}{\partial\widetilde{y_h}}\big)^2=4\frac{1}{(1+\gamma)^2}\widetilde{x_i}^2\widetilde{y_h}^2
  d^{-8-4\gamma}(\widetilde{z},0).
\end{equation}
\begin{equation}\label{37}
  \big(\frac{\partial y_j}{\partial\widetilde{x_k}}\big)^2=4\frac{1}{(1+\gamma)^2}\widetilde{y_j}^2\widetilde{x_k}^2
  |\widetilde{x}|^{4\gamma}d^{-8-8\gamma}(\widetilde{z},0).
\end{equation}
\begin{flalign}\label{38}
\big(\frac{\partial y_j}{\partial\widetilde{y_h}}\big)^2=
  \begin{cases}
d^{-4-4\gamma}(\widetilde{z},0)+4\widetilde{y_j}^4d^{-8-8\gamma}(\widetilde{z},0)-4\widetilde{y_j}^2
d^{-6-6\gamma}(\widetilde{z},0),
    &h=j,\\
 4\widetilde{y_j}^2\widetilde{y_h}^2d^{-8-8\gamma}(\widetilde{z},0), & h\neq j.
  \end{cases}
\end{flalign}
From \eqref{35} and \eqref{36}, we know that
\begin{equation}\label{39}
  \sum^N_{k=1}\big(\frac{\partial x_i}{\partial\widetilde{x_k}}\big)^2+|\widetilde{x}|^{2\gamma}\sum^l_{h=1}\big(\frac{\partial x_i}{\partial\widetilde{y_h}}\big)^2=d^{-4}(\widetilde{z},0).
\end{equation}
From \eqref{37} and \eqref{38}, we know that
\begin{equation}\label{40}
  \sum^N_{k=1}\big(\frac{\partial y_j}{\partial\widetilde{x_k}}\big)^2+|\widetilde{x}|^{2\gamma}\sum^l_{h=1}\big(\frac{\partial y_j}{\partial\widetilde{y_h}}\big)^2=|\widetilde{x}|^{2\gamma}d^{-4-4\gamma}(\widetilde{z},0).
\end{equation}
From \eqref{29}-\eqref{31}, we know that
\begin{equation}\label{41}
  \sum^N_{k=1}\frac{\partial^2x_i}{\partial\widetilde{x_k}^2}+|\widetilde{x}|^{2\gamma}
  \sum^l_{h=1}\frac{\partial^2x_i}{\partial\widetilde{y_h}^2}
  =\frac{2(2-N-l-l\gamma)}{(1+\gamma)^2}\widetilde{x_i}|\widetilde{x}|^{2\gamma}d^{-4-2\gamma}(\widetilde{z},0).
\end{equation}
From \eqref{32}-\eqref{34}, we know that
\begin{equation}\label{42}
  \sum^N_{k=1}\frac{\partial^2y_j}{\partial\widetilde{x_k}^2}+|\widetilde{x}|^{2\gamma}
  \sum^l_{h=1}\frac{\partial^2y_j}{\partial\widetilde{y_h}^2}
  =\frac{2(2-N-l-l\gamma)}{1+\gamma}\widetilde{y_j}|\widetilde{x}|^{2\gamma}d^{-4-4\gamma}(\widetilde{z},0).
\end{equation}
From \eqref{25}-\eqref{28}, we know that
\begin{equation}\label{43}
\begin{aligned}
  &\sum^N_{k=1}\frac{\partial x_i}{\partial\widetilde{x_k}}\frac{\partial x_{i'}}{\partial\widetilde{x_k}}+|\widetilde{x}|^{2\gamma}\sum^l_{h=1}\frac{\partial x_i}{\partial\widetilde{y_h}}\frac{\partial x_{i'}}{\partial\widetilde{y_h}}
  \\=&\sum^N_{k=1}\frac{\partial y_j}{\partial\widetilde{x_k}}\frac{\partial y_{j'}}{\partial\widetilde{x_k}}+|\widetilde{x}|^{2\gamma}\sum^l_{h=1}\frac{\partial y_j}{\partial\widetilde{y_h}}\frac{\partial y_{j'}}{\partial\widetilde{y_h}}
  \\=&\sum^N_{k=1}\frac{\partial x_i}{\partial\widetilde{x_k}}\frac{\partial y_j}{\partial\widetilde{x_k}}+|\widetilde{x}|^{2\gamma}\sum^l_{h=1}\frac{\partial x_i}{\partial\widetilde{y_h}}\frac{\partial y_j}{\partial\widetilde{y_h}}
  \\=&0.
\end{aligned}
\end{equation}
Thus according to \eqref{18}, \eqref{23}, \eqref{39}-\eqref{43}, we have calculated the first term in \eqref{7} as follows.
\begin{equation}\label{44}
\begin{aligned}
  &d^b(z(\widetilde{z}),0)\Delta_\gamma^{(\widetilde{z})}u(z(\widetilde{z}))
  \\=&d^{4+b}(z,0)\Delta_\gamma u(z)
  \\+&\sum^N_{i=1}\frac{\partial u}{\partial x_i}\left(\frac{2(2-N-l-l\gamma)}{(1+\gamma)^2}\widetilde{x_i}|\widetilde{x}|^{2\gamma}
  d^{-4-b-2\gamma}(\widetilde{z},0)\right)
  \\+&\sum^l_{j=1}\frac{\partial u}{\partial y_j}\left(\frac{2(2-N-l-l\gamma)}{1+\gamma}\widetilde{y_j}|\widetilde{x}|^{2\gamma}
  d^{-4-b-4\gamma}(\widetilde{z},0)\right).
\end{aligned}
\end{equation}

Finally, we will calculate the last term in \eqref{7} and will find the coefficients of first derivatives.\\
By \eqref{11}-\eqref{12}, \eqref{19}-\eqref{20}, \eqref{25}-\eqref{28}, we have
\begin{equation}\label{45}
\begin{aligned}
     &2\left(\nabla_{\widetilde{x}}u(z(\widetilde{z}))\cdot\nabla_{\widetilde{x}}d^b(z(\widetilde{z}),0)
  +|\widetilde{x}|^{2\gamma}
  \nabla_{\widetilde{y}}u(z(\widetilde{z}))\cdot\nabla_{\widetilde{y}}d^b(z(\widetilde{z}),0)\right)
  \\=&2\sum^N_{i=1}\frac{\partial u}{\partial x_i}\left(\sum^N_{k=1}\frac{\partial x_i}{\partial\widetilde{x_k}}
  \frac{\partial d^{-b}(\widetilde{z},0)}{\partial\widetilde{x_k}}+|\widetilde{x}|^{2\gamma}\sum^l_{h=1}
  \frac{\partial x_i}{\partial\widetilde{y_h}}\frac{\partial d^{-b}(\widetilde{z},0)}{\partial\widetilde{y_h}}\right)
  \\+&2\sum^l_{j=1}\frac{\partial u}{\partial y_j}\left(\sum^N_{k=1}\frac{\partial y_j}{\partial\widetilde{x_k}}
  \frac{\partial d^{-b}(\widetilde{z},0)}{\partial\widetilde{x_k}}+|\widetilde{x}|^{2\gamma}\sum^l_{h=1}
  \frac{\partial y_j}{\partial\widetilde{y_h}}\frac{\partial d^{-b}(\widetilde{z},0)}{\partial\widetilde{y_h}}\right)
   \\=&2\sum^N_{i=1}\frac{\partial u}{\partial x_i}\left(b\frac{1}{(1+\gamma)^2}\widetilde{x_i}|\widetilde{x}|^{2\gamma}d^{-4-b-2\gamma}(\widetilde{z},0)\right)
    \\+&2\sum^l_{j=1}\frac{\partial u}{\partial y_j}\left(b\frac{1}{1+\gamma}\widetilde{y_j}|\widetilde{x}|^{2\gamma}d^{-4-b-4\gamma}(\widetilde{z},0)\right).
\end{aligned}
\end{equation}
Put $b=N+(1+\gamma)l-2$ into \eqref{44} and \eqref{45} respectively and plus them, we obtain
\begin{equation}\label{46}
  \Delta_\gamma^{(\widetilde{z})}w(\widetilde{z})=d^{N+(1+\gamma)l+2}(z,0)\Delta_\gamma u(z)=d^{N_\gamma+2}(z,0)\Delta_\gamma u(z).
\end{equation}
In conclusion, we have proved Theorem \ref{th2}.
\\
\end{proof}

Next we give an application of Kelvin transformation. We consider the following equation
\begin{equation}\label{47}
  -\Delta_\gamma u=|u|^{p-1}u,\;\;z\in\R^{N+l}.
\end{equation}
Then under the Kelvin transformation, ~$w(\widetilde{z})$~satisfies
\begin{equation}\label{48}
  -\Delta_\gamma^{(\widetilde{z})}w(\widetilde{z})=d^{(N+(1+\gamma)l-2)p-(N+(1+\gamma)l)-2}
  (\widetilde{z},0)|w(\widetilde{z})|^{p-1}w(\widetilde{z}),\;\widetilde{z}\in\R^{N+l}\setminus\{0\}.
\end{equation}
Especially, if we take $p=2_\gamma^*-1$, where $2_\gamma^*=\frac{2N_\gamma}{N_\gamma-2}$ and $N_\gamma=N+(1+\gamma)l$, then~$w(\widetilde{z})$~satisfies
\begin{equation}\label{49}
  -\Delta_\gamma^{(\widetilde{z})}w(\widetilde{z})=
  |w(\widetilde{z})|^{2_\gamma^*-2}w(\widetilde{z}),\;\widetilde{z}\in\R^{N+l}\setminus\{0\}.
\end{equation}
Now we prove that $w(\widetilde{z})$ actually satisfies
\begin{equation}\label{50}
  -\Delta_\gamma^{(\widetilde{z})}w(\widetilde{z})=
  |w(\widetilde{z})|^{2_\gamma^*-2}w(\widetilde{z}),\;\widetilde{z}\in\R^{N+l}.
\end{equation}

\begin{proof}[\textbf{Proof of Corollary \ref{th3}:}]
First, we calculate that
\begin{equation}\label{51}
\begin{aligned}
   &\int_{\R^{N+l}}|w(\widetilde{z})|^{2_\gamma^*}d\widetilde{z}
\\\leq&\int_{\R^{N+l}}|d^{N_\gamma-2}(z,0)u(z)|^{\frac{2N_\gamma}{N_\gamma-2}}
\left(\frac{1}{d^2(z,0)}\right)^N\left(\frac{1}{d^{2+2\gamma}(z,0)}\right)^ldz
\\=&\int_{\R^{N+l}}|u(z)|^{2_\gamma^*}dz.
\end{aligned}
\end{equation}
Take $\xi(\widetilde{z})=\xi(|\widetilde{z}|)\in C^1(\R^{N+l})$ satisfying $\xi(\widetilde{z})=0$ in $B_1(0)$, $\xi(\widetilde{z})=1$ in $\big(B_2(0)\big)^c$, and $\xi(\widetilde{z})$ is non-decreasing on $|\widetilde{z}|$ in $B_2(0)\setminus B_1(0)$, where $\widetilde{z}=(\widetilde{x},\widetilde{y})$, $|\widetilde{z}|=\big(\frac{1}{(1+\gamma)^2}|\widetilde{x}|^{2+2\gamma}+|\widetilde{y}|^2\big)^\frac{1}{2+2\gamma}$ and $B_r(0)=\left\{\widetilde{z}\in\R^{N+l}|\;\big(\frac{1}{(1+\gamma)^2}
|\widetilde{x}|^{2+2\gamma}+|\widetilde{y}|^2\big)^\frac{1}{2+2\gamma}<r\right\}$.
Denote $\xi_h(\widetilde{z})=\xi\left(\frac{\widetilde{z}}{h^\frac{N+l}{N+(1+\gamma)l}}\right)$. For any $\varphi(\widetilde{z})\in C_0^\infty(\R^{N+l})$, we take $\xi_h^2(\widetilde{z})\varphi(\widetilde{z})$ as test function.
It follows \eqref{49} that
\begin{equation}\label{52}
  \int_{\R^{N+l}}\nabla_\gamma^{(\widetilde{z})}w(\widetilde{z})\cdot\nabla_\gamma^{(\widetilde{z})}
  \big(\xi_h^2(\widetilde{z})\varphi(\widetilde{z})\big)d\widetilde{z}
  =\int_{\R^{N+l}}|w(\widetilde{z})|^{2_\gamma^*-2}w(\widetilde{z})\xi_h^2(\widetilde{z})\varphi(\widetilde{z})
  d\widetilde{z}.
\end{equation}
Especially, we substitute $w(\widetilde{z})\varphi^2(\widetilde{z})$ for $\varphi(\widetilde{z})$ and obtain
\begin{equation}\label{53}
  \int_{\R^{N+l}}\nabla_\gamma^{(\widetilde{z})}w(\widetilde{z})\cdot\nabla_\gamma^{(\widetilde{z})}
  \big(\xi_h^2(\widetilde{z})w(\widetilde{z})\varphi^2(\widetilde{z})\big)d\widetilde{z}
  =\int_{\R^{N+l}}|w(\widetilde{z})|^{2_\gamma^*}\xi_h^2(\widetilde{z})\varphi^2(\widetilde{z})d\widetilde{z}.
\end{equation}
Since $\xi_h(\widetilde{z})\rightarrow1\;a.e.$ when $h\rightarrow0$, by Lebesgue's dominated convergence theorem we have
\begin{equation}\label{54}
  \lim_{h\rightarrow0}\int_{\R^{N+l}}|w(\widetilde{z})|^{2_\gamma^*}\xi_h^2(\widetilde{z})\varphi^2(\widetilde{z})
  d\widetilde{z}
  =\int_{\R^{N+l}}|w(\widetilde{z})|^{2_\gamma^*}\varphi^2(\widetilde{z})d\widetilde{z}.
\end{equation}
Then according to Cauchy inequality we have
\begin{equation}\label{55}
\begin{aligned}
  &\int_{\R^{N+l}}\nabla_\gamma^{(\widetilde{z})}w(\widetilde{z})\cdot\nabla_\gamma^{(\widetilde{z})}
  \big(\xi_h^2(\widetilde{z})w(\widetilde{z})\varphi^2(\widetilde{z})\big)d\widetilde{z}
\\=&\int_{\R^{N+l}}|\nabla_\gamma^{(\widetilde{z})} w(\widetilde{z})|^2\xi_h^2(\widetilde{z})\varphi^2(\widetilde{z})d\widetilde{z}
+2\int_{\R^{N+l}}\nabla_\gamma^{(\widetilde{z})} w(\widetilde{z})\xi_h(\widetilde{z})\varphi(\widetilde{z})w(\widetilde{z})
\nabla_\gamma^{(\widetilde{z})}\big(\xi_h(\widetilde{z})\varphi(\widetilde{z})\big)d\widetilde{z}
\\\geq&\frac{1}{2}\int_{\R^{N+l}}|\nabla_\gamma^{(\widetilde{z})} w(\widetilde{z})|^2\xi_h^2(\widetilde{z})\varphi^2(\widetilde{z})d\widetilde{z}
-C\int_{\R^{N+l}}w^2(\widetilde{z})|\nabla_\gamma^{(\widetilde{z})} \big(\xi_h(\widetilde{z})\varphi(\widetilde{z})\big)|^2d\widetilde{z}.
\end{aligned}
\end{equation}
Since
\begin{equation}\label{56}
\begin{aligned}
 &\int_{\R^{N+l}}w^2(\widetilde{z})|\nabla_\gamma^{(\widetilde{z})} \big(\xi_h(\widetilde{z})\varphi(\widetilde{z})\big)|^2d\widetilde{z}
\\=& \int_{B_{2h}(0)}w^2(\widetilde{z})|\nabla_\gamma^{(\widetilde{z})} \big(\xi_h(\widetilde{z})\varphi(\widetilde{z})\big)|^2d\widetilde{z}
+\int_{{\R^{N+l}}\setminus{B_{2h}(0)}}w^2(\widetilde{z})|\nabla_\gamma^{(\widetilde{z})} \big(\xi_h(\widetilde{z})\varphi(\widetilde{z})\big)|^2d\widetilde{z}
\\\leq&C\frac{1}{h^\frac{2(N+l)}{N+(1+\gamma)l}}\int_{B_{2h}(0)}w^2(\widetilde{z})d\widetilde{z}
+\int_{{\R^{N+l}}\setminus{B_{2h}(0)}}w^2(\widetilde{z})|\nabla_\gamma^{(\widetilde{z})} \varphi(\widetilde{z})|^2d\widetilde{z}
\\\leq&C\left(\int_{B_{2h}(0)}w^{2^*_\gamma}(\widetilde{z})d\widetilde{z}\right)^{\frac{2}{2^*_\gamma}}
+\int_{\R^{N+l}}w^2(\widetilde{z})|\nabla_\gamma^{(\widetilde{z})} \varphi(\widetilde{z})|^2d\widetilde{z}
\leq C,
\end{aligned}
\end{equation}
in which $\left(\displaystyle\int_{B_{2h}(0)}1d\widetilde{z}\right)^{1-\frac{2}{2^*_\gamma}}=(2h)^{\frac{2(N+l)}{N+(1+\gamma)l}}$
, $C$ is independent of $h$.\\
From \eqref{51}, \eqref{53}-\eqref{56} we know that
\begin{equation}\label{57}
  \int_{B_R(0)}|\nabla_\gamma^{(\widetilde{z})} w(\widetilde{z})|^2\leq C,\;\;\;\;\forall R>0.
\end{equation}
Look at the LHS of \eqref{52},
\begin{equation}\label{58}
\begin{aligned}
  &\int_{\R^{N+l}}\nabla_\gamma^{(\widetilde{z})}w(\widetilde{z})\cdot\nabla_\gamma^{(\widetilde{z})}
  \big(\xi_h^2(\widetilde{z})\varphi(\widetilde{z})\big)d\widetilde{z}
  \\=&2\int_{\R^{N+l}}\nabla_\gamma^{(\widetilde{z})}w(\widetilde{z})\xi_h(\widetilde{z})\nabla_\gamma^{(\widetilde{z})}
\xi_h(\widetilde{z})\varphi(\widetilde{z})d\widetilde{z}+\int_{\R^{N+l}}\xi_h^2(\widetilde{z})
\nabla_\gamma^{(\widetilde{z})}w(\widetilde{z})\cdot\nabla_\gamma^{(\widetilde{z})}\varphi(\widetilde{z})d\widetilde{z}.
\end{aligned}
\end{equation}
According to \eqref{57}, we see that as $h\rightarrow0$,
\begin{equation}\label{59}
\begin{aligned}
  &\left|\int_{\R^{N+l}}\nabla_\gamma^{(\widetilde{z})}w(\widetilde{z})\xi_h(\widetilde{z})\nabla_\gamma^{(\widetilde{z})}
\xi_h(\widetilde{z})\varphi(\widetilde{z})d\widetilde{z}\right|
\\\leq&C\frac{1}{h^\frac{N+l}{N+(1+\gamma)l}}\int_{B_{2h}(0)}|\nabla_\gamma^{(\widetilde{z})}w(\widetilde{z})|d\widetilde{z}
\\\leq&Ch^{(N+l)(\frac{1}{2}-\frac{1}{N_\gamma})}\left(\int_{B_{2h}(0)}|\nabla_\gamma^{(\widetilde{z})}w(\widetilde{z})|^2
d\widetilde{z}\right)^\frac{1}{2}\rightarrow0,
\end{aligned}
\end{equation}
where $(N+l)(\frac{1}{2}-\frac{1}{N_\gamma})>0$ since $N+l>3$.\\
Thus, \eqref{52}, \eqref{58} and \eqref{59} give us
\begin{equation}\label{60}
  \int_{\R^{N+l}}\xi_h^2(\widetilde{z})
\nabla_\gamma^{(\widetilde{z})}w(\widetilde{z})\cdot\nabla_\gamma^{(\widetilde{z})}\varphi(\widetilde{z})d\widetilde{z}
+o_h(1)=\int_{\R^{N+l}}|w(\widetilde{z})|^{2_\gamma^*-2}w(\widetilde{z})\xi_h^2(\widetilde{z})\varphi(\widetilde{z})
  d\widetilde{z},\;
   \forall\varphi(\widetilde{z})\in C_0^\infty(\R^{N+l}).
\end{equation}
Letting $h\rightarrow0$ in \eqref{60}, finally we get
\begin{equation}\label{61}
  \int_{\R^{N+l}}
\nabla_\gamma^{(\widetilde{z})}w(\widetilde{z})\cdot\nabla_\gamma^{(\widetilde{z})}\varphi(\widetilde{z})d\widetilde{z}
=\int_{\R^{N+l}}|w(\widetilde{z})|^{2_\gamma^*-2}w(\widetilde{z})\varphi(\widetilde{z})
  d\widetilde{z},\;
   \forall\varphi(\widetilde{z})\in C_0^\infty(\R^{N+l}).
\end{equation}
This is the definition of weak solution of equation \eqref{50}.
\end{proof}

\textbf{Acknowledgments}

The authors are grateful to the referees for their careful reading and valuable comments.
\\*

\textbf{Data Availability Statement}

No data, models, or code were generated or used during this study.
\\*

\textbf{Conflict of interest statement}

The authors declare that they have no conflicts of interest in the research presented in this manuscript.
\\*

\end{document}